\documentclass[12pt]{amsart}
\usepackage{fullpage,url,amssymb,bm}
\usepackage[all]{xy} 
\DeclareFontEncoding{OT2}{}{} 

\font\elevensc=cmcsc10 scaled\magstephalf

\font\teneu=eufm10 scaled\magstep1\font\seveneu=eufm7 scaled\magstep1
                              \font\fiveeu=eufm5 scaled\magstep1
\font\tenlv=msbm10 scaled\magstep1\font\sevenlv=msbm7 scaled\magstep1
                              \font\fivelv=msbm5 scaled\magstep1

\newfam\eufam  \def\eu{\fam\eufam\teneu}
\textfont\eufam=\teneu \scriptfont\eufam=\seveneu
                              \scriptscriptfont\eufam=\fiveeu

\newfam\lvfam  \def\lv{\fam\lvfam\tenlv}
\textfont\lvfam=\tenlv \scriptfont\lvfam=\sevenlv
                              \scriptscriptfont\lvfam=\fivelv
\newcommand{\defi}[1]{\textsf{#1}}
\newcommand{\clA}{{\mathcal A}}
\newcommand{\clB}{{\mathcal B}}

\newcommand{\clK}{{\mathcal K}}
\newcommand{\clL}{{\mathcal L}}

\newcommand{\clO}{{\mathcal O}}

\newcommand{\clR}{{\mathcal R}}
\newcommand{\clS}{{\mathcal S}}

\newcommand{\clX}{{\mathcal X}}
\newcommand{\clY}{{\mathcal Y}}

\newcommand{\euD}{{\eu D}}

\newcommand{\euX}{{\eu X}}
\newcommand{\euY}{{\eu Y}}

\newcommand{\lvA}{{\lv  A}}

\newcommand{\lvF}{{\lv F}}

\newcommand{\lvP}{{\lv P}}

\newcommand{\lvZ}{{\lv Z}}

\newcommand{\eum}{{\eu m}}

\newcommand{\euo}{{\eu o}}
\newcommand{\eup}{{\eu p}}
\newcommand{\euq}{{\eu q}}
\newcommand{\eur}{{\eu r}}

\DeclareMathOperator{\chr}{char}

\DeclareMathOperator{\Proj}{Proj}

\DeclareMathOperator{\Spec}{Spec}

\newtheorem{theorem}{\indent Theorem}[section]

\newtheorem{keylemma}[theorem]{\indent Key Lemma}
\newtheorem{lemma}[theorem]{\indent Lemma}
\newtheorem{LGP}[theorem]{\indent LGP}
\newtheorem{proposition}[theorem]{\indent Proposition}

\theoremstyle{definition}

\newtheorem{definition}[theorem]{\indent Definition}

\newtheorem{example/fact}[theorem]{\indent Example/Fact}
\newtheorem{fact}[theorem]{\indent Fact}

\newtheorem{fact/definition}[theorem]{\indent Fact/Definition}

\newtheorem{Gal-criterion}[theorem]{\indent {\rm Gal}-Criterion}

\newtheorem{remark/definition}[theorem]{\indent Remark/Definition}
\newtheorem{remark/notation}[theorem]{\indent Remark/Notation}
\newtheorem{definition/remark}[theorem]{\indent Definition/Remark}
\newtheorem{notations}[theorem]{\indent Notations}
\newtheorem{notations/facts}[theorem]{\indent Notations/Facts}
\newtheorem{remark}[theorem]{\indent Remark}

\newtheorem{hypo}[theorem]{\indent Hypothesis}

\newcommand{\mbl}{
\newcommand{\mbl}{\baselineskip20pt\lineskip3pt\lineskiplimit3pt}}

\newcommand{\dwn}[1]{\phantom{\Big|}\hhb{-6}%
\big\downarrow\rlap{$\vcenter{\hbox{$\scriptstyle{\!#1}$}}$}}

\newcommand{\hhb}[1]{\hbox to#1pt{}}
  \newcommand{\hor}[1]{\smash
     {\mathop{{\lgrghtar}}\limits^{\lower2pt\hbox{$\scriptstyle{#1}$}}}}
\newcommand{\horr}[1]{\smash
     {\mathop{{\lglgrghtar}}\limits^{\lower2pt\hbox{$\scriptstyle{#1}$}}}}
\newcommand{\horrr}[1]{\smash
     {\mathop{{\lglglgrghtar}}\limits^{\lower2pt\hbox{$\scriptstyle{#1}$}}}}

\newcommand{\ilim}[1]{\hbox to14pt{lim\kern-14pt\lower4.5pt%
\hbox{$\scriptstyle\longrightarrow$}%
\kern-8pt\lower8.5pt\hbox{$\scriptstyle{#1}$}}\hhb{3}}

\newcommand{\lgrghtar}{{\hhb2{\relbar\joinrel\rightarrow}\hhb2}}
\newcommand{\lglgrghtar}{{\hhb1{\relbar\joinrel\relbar\joinrel\rightarrow}\hhb1}}
\newcommand{\lglglgrghtar}{{\hhb1{\relbar\joinrel\relbar\joinrel%
\relbar\joinrel\rightarrow}\hhb1}}
\newcommand{\nmnm}[1]{{\elevensc #1}} 
\newcommand{\nmi}{{\phantom{'}}}

\newcommand{\oli}{\overline}

\newcommand{\onto}{{\hhb1\to\hhb{-16}\to\hhb0}}
\newcommand{\plim}[1]{\hbox to14pt{lim\kern-14pt\lower4.5pt%
\hbox{$\scriptstyle\longleftarrow$}\kern-8pt\lower8.5pt\hbox{$\scriptstyle{#1}$}}\hhb{3}}
\newcommand{\ratmap}{{\hhb2\hbox{-\hhb{1}-\hhb{1}-}\hhb{-1}%
\lower-1pt\hbox{$\scriptscriptstyle>$}\hhb2}}
\newcommand{\rsdp}{{\,\times\kern-3pt\lower-1pt%
\hbox{$\scriptscriptstyle|$\hhb3}}}
\newcommand{\sdp}{\hbox to10pt{\hss\hbox{\mathsurround=0pt$\times$\kern-1.6pt
       \hbox{\vrule height5.2pt width.6pt}\hbox to1.6pt{}}\hss}}

\newcommand{\vid}{{\hhb{-3.25}\not\hhb{-2.5}\lower-1.25pt%
\hbox{\mathsurround=0pt$\scriptscriptstyle\bigcirc$}}}
\newcommand{\vvidd}{{\lower-1pt\hbox{/}\kern-7pt{\hbox{O}}}}
\font\mrrm=cmr12 scaled\magstephalf

\newcommand{\alp}{\nu}

\newcommand{\alrh}{{\rho}}

\newcommand{\akk}{{l}}

\newcommand{\clOy}{{\clO_{\hhb{-.25}y}}}
\newcommand{\clOq}{{\clO_{\hhb{-.25}\euq}}}
\newcommand{\clOz}{{\clO_{\hhb{-.25}\eup}}}

\newcommand{\dgr}[1]{{\rm deg}\hbox{\mrrm(}#1\hbox{\mrrm)}}

\newcommand{\eupq}{{\euq}}
\newcommand{\eupz}{{\eup}}
\newcommand{\euox}{{\euo_x}}
\newcommand{\euoy}{{\euo_y}}

\newcommand{\hra}{\hookrightarrow}

\newcommand{\ix}{\imath}

\newcommand{\jx}{\jmath}

\newcommand{\kk}{{\hat\kappa}} 
 
\newcommand{\kbm}{{\bm k}} 
\newcommand{\kp}[1]{{\kappa_{#1}}}

\newcommand{\kzr}{{\lvF}}
\newcommand{\llll}{{\bm\lambda}}

\newcommand{\lps}[1]{(\hhb{-1}(#1)\hhb{-1})}
\newcommand{\lvAix}{{\lvA^{\hhb{-1}{\scriptscriptstyle|}\uix{\scriptscriptstyle|}}}}
\newcommand{\lvPix}{{\lvP^{{\scriptscriptstyle|}
                          \uix{\scriptscriptstyle|}}}}
\newcommand{\lvAixZ}{{\lvA^{\hhb{-1}
    {\scriptscriptstyle|}\uix{\scriptscriptstyle|}}_{\scriptscriptstyle\Srz}}}
\newcommand{\lvPixZ}{{\lvP^{{\scriptscriptstyle|}\uix
    {\scriptscriptstyle|}}_{\scriptscriptstyle\Srz}}}

\newcommand{\lvPt}[1]{\lvP^1_{\hhb{-1}t\hhb{0},\hhb{.5}#1}}
\newcommand{\lvAt}[1]{\lvA^1_{\hhb{.5}t\hhb{0},\hhb{.5}#1}}

\newcommand{\mPi}{{\textstyle\prod}}

\newcommand{\mx}{f}

\newcommand{\nx}{e}

\newcommand{\olix}{{\overline x}}

\newcommand{\piz}{\pi_0}

\newcommand{\rhal}{{\hhb{.5}\rho}}

\newcommand{\Srp}{{\bm Z}}

\newcommand{\Srz}{{\bm Z}_0}
\newcommand{\Scz}{{\scriptscriptstyle{\bm Z}_0}}

\newcommand{\scl}[1]{#1^{\scriptscriptstyle{\rm sep}}}

\newcommand{\shra}[1]{\!\!\!\hookrightarrow\!\!\!}

\newcommand{\micsm}[2]{\mathop{\textstyle\sum}\limits_{#1}^{#2}}

\newcommand\Spf{{\rm Spf}}

\newcommand{\str}[2]{{{}^*\hhb{-#1}#2}}

\newcommand{\TT}{{t}}
\newcommand{\TTmu}{\TT_\mu^{-1}}

\newcommand{\tht}{\theta}

\newcommand{\tlclK}{\clK^{^{\rm nr}}}
\newcommand{\tlclL}{\clL^{^{\rm nr}}}
\newcommand{\tlclO}{{\tilde\clO}}

\newcommand{\tlV}{{\tilde V}}

\newcommand{\tmu}{{t^{-1}}}

\newcommand{\ultrb}[1]{{\big(\mPi_x #1\big)/\hhb1{\eu U}}}

\newcommand{{\upi}}{{\varpi}}

\newcommand\uix{{\undr\ix}}

\newcommand{\ux}{\imath}

\newcommand{\undr}[1]{{\bm #1}}

\newcommand{\V}{V}
\newcommand{\W}{W}

\newcommand{\vK}{v\!_K^\nmi}
\newcommand{\vL}{v\!_L^\nmi}

\newcommand{\Wnx}{W_{\!\nx}}
\newcommand{\Wstr}{{\str1W}}

\newcommand{\whclO}{{\widehat{\clO}}}

\newcommand{\wstr}{{\str1w}}

\newcommand{\xbf}{{x}}

\newcommand{\y}{y}
\newcommand{\ybf}{{y}}

\title[Lifting of curves]{$\hbox to\hsize{\hss$\scriptstyle
{\eu ...meiner\hhb3 lieben\hhb3 Gerda\hhb3 gewidmet}$}$
\vskip25pt
{\Large Lifting} ${\bm{o\!f}}$ {\Large Curves}
\vskip10pt
 ---\ ${\bf 
 T\hhb1h\hhb1e\hhb1\ O\hhb1o\hhb1r\hhb1t\ 
   C\hhb1o\hhb1n\hhb1j\hhb1e\hhb1c\hhb1t\hhb1u\hhb1r\hhb1e}$ ---}
\author{Florian Pop}
\address{Department of Mathematics 
        \vskip0pt
        University of Pennsylvania
        \vskip0pt
        DRL, 209 S 33rd Street
        \vskip0pt
        Phila\-delphia, PA 19104. USA}
\email{pop@math.upenn.edu}
\urladdr{http://math.penn.edu/\~{}pop}

\usepackage[
        backref,
        pdfauthor={Florian Pop}, 
        pdftitle={Lifting of curves},
]{hyperref}
\usepackage[backrefs,lite]{amsrefs} 

\begin{document}
%
%
\begin{abstract}
In this note we show that a special case of a recent result 
by Obus--Wewers (used as a black box) together with a 
deformation argument in characteristic $p$ leads to a proof 
of the Oort Conjecture in the general case. A boundedness
result is given as well.
\end{abstract}

                  
\thanks{Supported by NSF grant DMS-1101397.}
\date{12.29.11 / 1.30.12 / 2.12.12}

\maketitle

%
%

  
%
%
\section{Introduction}
The aim of this note is to present a proof of the (classical)
Oort Conjecture, which is a question about lifting Galois 
covers of curves from characteristic~$p>0$ to characteristic 
zero. In one form or the other, this kind of question might 
well be considered math folklore, and it was also well known 
that in general the lifting  is not possible. The problem was 
systematically addressed and formulated by \nmnm{Oort} 
at least as early as~1987, see~\cite{Oo1} or rather~\cite{Oo2}. 
The general context of the lifting question/problem is as 
follows: Let $k$ be an algebraically closed field of characteristic 
$p>0$, and $W(k)$ be the ring of Witt vectors over~$k$. Let 
$Y \to X$ a $G$-cover of complete smooth $k$-curves, where
$G$ is a finite group. Then the question is whether there 
exists a finite extension of discrete valuation rings $W(k)\hra R$ 
and a $G$-cover $\clY_R\to\clX_R$ of complete smooth 
$R$-curves whose special fiber is the given $G$-cover 
$Y\to X$. The answer to this question in general is negative, 
because over $k$ there are curves of genus $g>1$ and huge 
automorphism groups, see e.g.~\nmnm{Roquette}~\cite{Ro2}, 
whereas in characteristic zero one has the Hurwitz bound $84(g-1)$ 
for the order of the automorphism group. The Oort Conjecture 
is about a subtle interaction between the ramification structure
and the nature of the inertia groups of generically Galois 
covers of curves, the idea being that if the inertia groups of 
a $G$-cover $Y\to X$ look like in characteristic zero, i.e., 
they are all cyclic, then the Galois cover should be smoothly
liftable to characteristic zero. 
\vskip7pt
{\bf Oort Conjecture}. {\it Let $k$ and $W(k)$ be as above.
Let $Y\to X$ be a possibly ramified 
$G$-cover of complete smooth $k$-curves having only 
cyclic groups as inertia groups. Then there exists a finite 
extension $R$ of $\,W(k)$ over which the $G$-cover $Y\to X$
lifts smoothly, i.e., there exists a $G$-cover $\clY_R\to\clX_R$ 
of complete smooth $R$-curves with special fiber $Y\to X$.\/}
\vskip7pt
There is also the {\bf local Oort conjecture}, which 
asserts that every finite cyclic extension $k[[t]]\hra k[[z]]$ 
is the canonical reduction of a cyclic extension 
$R[[T]]\hra R[[Z]]$ for some finite extension~$R$ of $W(k)$.
The local/global Oort Conjecrures are related as follows,
see~Fact~\ref{equivOC}, where further equivalent forms 
of the Oort Conjecture are given: Let $R$ be a 
finite extension of $W(k)$, and $\clX_R$ be a complete
smooth curve with special fiber $X$. Then a given $G$-cover 
$Y\to X$, $y\mapsto x$, with cyclic inertia groups lifts 
smoothly over a given $R$ iff the local cyclic 
extensions $k[[t_x]]:=\whclO_{X,x}\hra\whclO_{Y,y}=:k[[t_y]]$ 
lift smoothly over~$R$ for all $x\in X$ and $y\mapsto x$.
\vskip4pt
{\bf Notation:} Let $\deg_p(\euD)$ be the power of~$p$ 
in the degree of the different of $k[[t]]\hra k[[z]]$, and 
$\deg_p(\euD_x)$ be correspondingly defined for the 
local extension $k[[t_x]]\hra k[[t_y]]$ at $y\mapsto x$.
%
\begin{theorem}
\label{OC}
The Oort Conjecture holds. Moreover, for every $\delta$ 
there exists an algebraic integer $\pi_\delta$ such that 
for every algebraically closed field $k$, $\chr(k)=p$, the
following hold:
\vskip2pt
{\rm1)} Let $k[[t]]\hra k[[z]]$ be a cyclic extension 
with $\deg_p(\euD)\leq\delta$.
\vskip2pt
{\rm2)} Let $Y\to X$ be a $G$-cover with cyclic inertia groups 
and $\deg_p(\euD_x)\leq\delta$ for all $x\in X$.
\vskip2pt
\noindent
Then $k[[t]]\hra k[[z]]$ and $Y\to X$ are smoothly liftable 
over $W(k)[\pi_\delta]$. 
\end{theorem}
%
{\bf Historical Note}: The first evidence for the Oort 
Conjecture is the fact that the conjecture holds for 
$G$-covers $Y\to X$ which have tame ramification only, 
i.e., $G$-covers whose inertia groups are cyclic of the 
form $\lvZ/m$ with $(p,m)=1$. Indeed, the lifting of 
such $G$-covers follows from the famous 
{\it Grothendieck's specialization theorem\/} for the 
tame fundamental group, see e.g.,~SGA~I. 
The first result which involved typical wild ramification 
is \nmnm{Oort--Sekiguchi--Suwa}~\cite{OSS} which 
tackled the case of $\lvZ/p$-covers. It was followed 
by a quite intensive research 
activity, see the survey article by \nmnm{Obus}~\cite{Ob} 
as well as the bibliography list at the end of this 
note. \nmnm{Garuti}~\cite{Ga1},~\cite{Ga2} contains a 
lot of foundational work and beyond that showed that 
every $G$-cover $Y\to X$ has (usually) non-smooth 
liftings with a well understood geometry. This aspect 
of the problem was revisited recently by \nmnm{Saidi}~\cite{Sa}, 
where among other things a systematic discussion 
of the (equivalent) forms of the Oort Conjecture is given. The 
paper~\cite{GM1} by \nmnm{Green--Matignon} contains 
further foundational work and gives a positive answer 
to the Oort Conjecture in the case of inertia groups of 
the form $\lvZ/mp^e$ with $(p,m)=1$ and $e\leq2$. 
Their result relies on the Sekiguchi--Suwa theory, 
see~\cite{SS1} and~\cite{SS2}. The paper 
\nmnm{Bertin--M\'ezard}~\cite{B-M} addresses the 
deformation theory for covers, whereas 
\nmnm{Chinburg--Guralnick--Harbater} \cite{CGH1}, 
\cite{CGH2} initiated the study of the so called {\it Oort 
groups,\/} and showed that the class of Oort groups 
is quite restrictive. Last but not least, the very recent 
result by \nmnm{Obus--Wewers}~\cite{O--W}, see
rather \nmnm{Obus}~\cite{Ob},~Theorem~6.28, solves 
the Oort Conjecture in the case the inertia groups are 
of the form $\lvZ/mp^e$ for $(p,m)=1$~and~$e\leq 3$, 
and {\it essential for the method of this note,\/} when 
the upper ramification jumps are subject to some 
explicit (strong) restrictions, see the explanations 
at~Remark~\ref{remarkOW} in section~4 for details. 
\vskip5pt
{\bf About the proof}: Concerning technical tools,
we use freely a few of the foundational results from the 
papers mentioned above. The main novel tools for the 
proof are Key Lemma~\ref{keylemma1} and its global
form~Theorem~\ref{charpOC}, and second,
Key~Lemma~\ref{keylemma2}, which is actually a very 
special case of Lemma~6.27 from~\nmnm{Obus}~\cite{Ob}, 
see section~4 for precise details and citations. All these
results will be used as ``black boxes'' in the proof of
the Oort Conjecture, given in section~4. Concerning the
idea of the proof, there is little to say: The point is to 
first deform a given $G$-cover $Y\to X$ to a cover 
$\clY_\euo\to\clX_\euo$ over $\euo=k[[\upi]]$ in such a 
way that the ramification of the deformed cover has no 
\defi{essential upper jumps} as defined/introduced at 
the beginning of section~3, then apply the local-global
principles, etc. 
\vskip2pt
Maybe it's interesting to mention that the first variant
of the proof (January, 2011) was shorter, but relied 
heavily on model theoretical tools and was not effective 
(concerning the finite extensions of $W(k)$ over which 
the smooth lifting can be realized). 
\vskip5pt
{\bf Acknowledgements(...)} If my recollection is correct, during 
an MFO Workshop in 2003~or~so, someone asked what
should be the ``characteristic $p$ Oort Conjecture,'' but
it seems that nobody ever followed up (successfully) on 
that idea.
\vskip5pt
\hhb{0}
%
%
%
\section{Reviewing well known facts}
Throughout this section, $k$ is an algebraically closed
field with ${\rm char}(k)=p>0$. All the other fields will be 
field extensions of $k$, in particular will be fields of
characteristic~$p$.
\vskip7pt
\noindent
A) {\it Reviewing higher ramification for cyclic extensions\/}
\vskip5pt
Let $K$ be a complete discrete valued field of positive
characteristic $p>0$, with valuation ring $R$ and having
as residue field an algebraically closed field $k$. Let 
$L|K$ be a finite Galois  extension, say with Galois 
group $G={\rm Gal}(L|K)$. Since $K$ is complete, the 
valuation $\vK$ of $K$ as a unique prolongation to $L$.
And since the residue field $k$ of $K$ is algebraically 
closed, $L|K$ is totally ramified, i.e., $[L:K]=e(L|K)$.
Finally, let $\vL$ be the normalized valuation of $L$, 
hence $\vL(L^\times)=Z$ and ${1\over[L|K]}\vL=\vK$ on $K$.
\vskip2pt
Recall that the lower ramification 
groups of $L|K$ or of $G$ are defined as follows: 
Let $z\in L$ be a uniformizing parameter. For every 
$\jmath$ we set $G_\jx:=\{\sigma\in G\mid\vL(\sigma z-z)>\jx\}$ 
and call it the $\jx^{\rm th}$ \defi{lower ramification group} 
of $L|K$ or of $G$. Clearly, $G=G_0$ is the inertia 
group of $L|K$, and $G_1$ is the ramification group 
of $L|K$, thus the Sylow $p$-group of $G$, and $G_\jx=1$ 
for $\jx$ sufficiently large. In particular, $G=G_1$ iff $G$ 
is a $p$-group iff $L|K$ is totally wildly ramified. 
\vskip2pt
The first important fact about the lower ramification groups 
$(G_\jx)_\jx$ is {\it Hilbert's different formula,\/} which gives an 
estimate for the degree $\dgr{\euD_{L|K}}$ of the different of
$L|K$ in terms of the orders of the lower ramification groups,
see e.g., \nmnm{Serre}~\cite{Se},~IV,~\S1:
\[
\deg(\euD_{L|K})=\sum_{\jx=0}^\infty\big(|G_\jx|-1\big).
\]
We further denote by $\jx_\alrh$ the \defi{lower jumps} 
for $L|K$, or of $G$, as being the numbers satisfying 
$G_{\jx_\alrh}\neq G_{\jx_\alrh+1}$. In particular, 
setting $\jx_{-1}=-1$ and $\jx_0=0$, and denoting the
upper jumps for~$L|K$ by $\jx_0\leq\jx_1\leq\dots\leq\jx_r$, 
one has that $\jx_r=\max\,\{\hhb2\jx \mid G_\jx\neq1\,\}$. 
\vskip2pt
Now suppose that $L|K$ is a cyclic extension with
$G=\lvZ/p^\nx$. Then $G=G_0=G_1$ by the discussion above,
and every subgroup of $G$ is a lower ramification group
for $L|K$. Thus one has precisely $\nx$ lower positive jumps 
$\jx_1\leq\dots\leq\jx_\nx$, and $G_{\jx_1}\geq\dots\geq G_{\jx_\nx}$
are precisely the $\nx$ non-trivial subgroups of $G=\lvZ/p^\nx$.
Finally, the Hilbert's Different Formula becomes:
\begin{eqnarray*} 
\textstyle
\deg(\euD_{L|K})=p^\nx-1+\sum_{\alrh=1}^\nx
(\jx_\alrh-\jx_{\alrh-1})\big(|G_{\jx_\alrh}|-1\big)
   &=&p^\nx-1+\sum_{\alrh=1}^\nx
              (\jx_\alrh-\jx_{\alrh-1})(p^{\nx-(\alrh-1)}-1)
\end{eqnarray*}

We recall that the lower ramification
subgroups behave functorially in the base field, i.e., 
if $K'|K$ is some finite sub-extension of $L|K$, and 
$G'\subseteq G$ is the Galois group of $L|K'$, then 
$G'_\jx=G_\jx\cap G'$. Since the lower ramification 
groups do not behave functorially with respect to 
Galois sub-extensions, one introduced the \defi{upper 
ramification groups} $G^{^{\,\scriptstyle\ix}}$ for 
$\ix\geq-1$ of $L|K$, which behave functorially under 
taking Galois sub-extensions, see \nmnm{Serre}~{Se},~IV,~\S3.
\vskip2pt
At least in the case of cyclic extensions $L|K$ with Gallois
group $G\cong\lvZ/p^\nx$, the formula which relates the 
lower ramification groups $G_{\!\jmath}$ to the upper ramification
groups $G^{^{\,\scriptstyle\ix}}$ is explicit via {\it Herbrand's 
formula,\/} see e.g.\ \nmnm{Serre}~\cite{Se},~IV,~\S3. 
Namely, if $\ux_0:=0$ and $\ux_1\leq \dots\leq \ux_\nx$ 
are the \defi{upper jumps} for $L|K$, then one has: 
\[
\jx_\alrh-\jx_{\alrh-1}=p^{\,\alrh-1}(\ux_\alrh-\ux_{\alrh-1}),
\quad \alrh=1,\dots,\nx.
\] 
\indent
Thus in the case $L|K$ is cyclic with Galois group
$G=\lvZ/p^\nx$, one can express the degree of the different
of $L|K$ in terms of higher ramification groups as follows:
\[
\deg(\euD_{L|K})=p^\nx-1+\sum_{\alrh=1}^\nx
(\ux_\alrh-\ux_{\alrh-1})p^{\alrh-1}(p^{\nx-(\alrh-1)}-1)=
\sum_{\alrh=1}^\nx(\ux_\alrh+1)(p^{\hhb1\alrh}-p^{\alrh-1}).
\]
\noindent
B) {\it Explicit formulas via Artin--Schreier--Witt Theory\/}
\vskip5pt
In the above notations, let $t$ be any uniformizing
parameter of the complete discrete valued field $K$
of characteristic $p > 0$. Then $K$ is canonically 
isomorphic to the Laurent power series field $K=k\lps t$ 
in the variable $t$ over $k$.
\vskip2pt
Recall that the Artin--Schreier--Witt theory gives a 
description of the cyclic $p$-power extensions of $K$ 
via finite length Witt vectors as follows, see e.g., 
\nmnm{Lang}~\cite{La}, or \nmnm{Serre}~\cite{Se},~II. Let 
$\clA$ be an integrally closed domain which is a $k$-algebra, 
and $\Wnx(\clA)=\{(a_1,\dots,a_\nx)\mid a_i\in\clA\}$
be the Witt vectors of length $\nx$ over $\clA$. 
Then the Frobenius morphism ${\rm Frob}$ of $\clA$
lifts to the \defi{Frobenius morphism} ${\rm Frob}_\nx$ 
of $\Wnx(\clA)$, and one defines the 
\defi{Artin--Schreier--Witt operator} 
$\wp_\nx:={\rm Frob}_\nx-{\rm Id}$ of 
$\clA$. If $\clA\hra\clA^{^{\rm nr}}$ is a ind-\'etale 
universal cover of $\clA$, one has the Artin--Schreier--Witt 
exact sequence 
\[
0\to \Wnx(\lvF_p)=\lvZ/p^\nx\hor{}\Wnx(\clA^{^{\rm nr}})
    \hor{\wp_\nx}\Wnx(\clA^{^{\rm nr}})\to0.
\]
of sheaves on ${\rm Et}(\clA)$. In particular, if 
${\rm Pic}(\clA)=0$, one gets a canonical isomorphism
\[
\Wnx(\clA)/{\rm im}(\wp_\nx)\to
     {\rm Hom}\big(\pi_1(\clA), \lvZ/p^\nx\big),
\]
which gives a canonical bijection between the cyclic 
subgroups $\langle{\undr a}\rangle\subset 
                                    \Wnx(\clA)/{\rm im}(\wp_\nx)$ 
and the integral \'etale cyclic extensions $\clA\hra\clA_{\undr a}$ 
with Galois group a quotient of $\lvZ/p^\nx$ by via
\[
\langle{\undr a}\rangle\mapsto \clA_{\undr a}:=\clA[{\undr x}]
\]
where ${\undr x}=(x_1,\dots,x_\nx)$ is any solution of the 
equation $\wp_\nx({\undr x})={\undr a}$.
\vskip2pt
In the special case $\clA=K=k\lps t$, one can make 
things more precise as follows. First, by Hensel's Lemma,
the class of every element in $\Wnx(K)/{\rm im}(\wp_\nx)$
contains a representative of the form ${\undr p}=(p_1,\dots,p_\nx)$
with $p_\alrh=p_\alrh(\tmu) \in k[\tmu] $ a polynomial 
in the variable $\tmu$ over $k$. Second, using the 
properties of the Artin--Schreier operator $\wp_\nx$, one 
can inductively ``reduce'' the terms of each $p_\alrh(\tmu)$ 
which contain powers of $\tmu$ to exponent divisible by 
$p$. If all the polynomials $p_\alrh(\tmu)$ have this property, 
one says that ${\undr p}=(p_1,\dots,p_\nx)$ is in \defi{standard 
form}. Following \nmnm{Garuti}~\cite{Ga3}, Thm.~1.1, and 
\nmnm{Thomas}~\cite{Tho},~Prop.~4.2, see also 
\nmnm{Obus--Priess}~\cite{O--P}, one can describe the 
upper jumps $\ux_1\leq\dots\leq \ux_\nx$ for the extension 
$L=K_{\undr p}$ with ${\undr p}=(p_1,\dots,p_\nx)$ in standard 
form and $p_1(\tmu)\neq0$ as follows:\footnote{Notice that 
we use the conventions $\deg(0)=-\infty$ and $\ix_0=0$.}
\[
\ux_\alrh=\max\,\{\hhb2p\hhb1\ux_{\alrh-1},\dgr{p_\alrh(\tmu)}\},
\quad \alrh=1,\dots,\nx.
\]
In particular, $\ux_\alrh\geq p^{\hhb1\alrh}\dgr{p_1(\tmu)} $, 
and the highest non-zero upper ramification index, 
i.e., the Artin conductor of $K_{\undr p}|K$ is
$\ux_\nx=\max\,\{\hhb2p^{\nx-\alrh}
             \dgr{p_\alrh(\tmu)}\mid \alrh=1,\dots,\nx\}$.
\vskip3pt
We make the following remark for later use: For 
${\undr a}=(a_1,\dots,a_\nx)$ an arbitrary Witt vector of 
length $\nx$ over $K$, and $K_{\undr a}|K$ as above one 
has: $[K_{\undr a}:K]=p^m$, where $m$ is minimal such 
that $(a_1,\dots,a_{e-m})\in{\rm im}(\wp_{e-m})$. In particular,
if $m<\nx$, then setting ${\undr b}=(b_1,\dots,b_m)$ with
$b_\alrh=a_{\alrh+\nx-m}$, one has: $K_{\undr a}|K$ is 
actually the cyclic extension $K_{\undr b}|K$ of degree 
$p^m$ of $K$, and one can compute the upper ramification 
indices of $K_{\undr a}=K_{\undr b}$ using the discussion 
above.  
\vskip7pt
\noindent
C) {\it Kato's smoothness criterion\/}
\vskip5pt
Let $k$ be as usual an algebraically closed field of 
characteristic $p>0$, and $\euo$ a complete discrete
valuation ring with quotient field $\kk={\rm Quot}(\euo)$
and residue field $k$. Let $\clA=\euo[[T]]$ be the power 
series ring over $\euo$. Then $\clR:=\clR\otimes_\euo\kk
=\kk\langle\hhb{-1}\langle T\rangle\hhb{-1}\rangle$
is the ring of power series in $T$ over $\kk$ having 
$v_\kk$-bounded coefficients. Thus $\clR$ is a Dedekind 
ring having $\Spec(\clR)$ in bijection with the points
of the open rigid disc $\euX=\Spf\,\clR$ of radius~$1$ 
over the complete valued field~$\kk$. Further, $\clA$ is 
a two dimensional complete regular ring with maximal 
ideal $(\pi,T)$ with $\clA\to\clA/(\pi,T)=k$, and 
$\euX=\Spec(\clR)$ is nothing but the complement 
of $V(\pi)\subset\Spec(\clA)$. Finally, 
$A:=\clA/(\pi)=k[[t]]$ is the power series ring in the 
variable $t:=T\,{\rm\big(mod}\,(\pi)\big)$, thus 
a complete discrete valuation ring. 
\vskip2pt
Let $\clK:={\rm Quot}(\clA)$ and $K:={\rm Quot}(A)=k\lps t$
be the fraction fields of $\clA$, respectively~$A$. Let 
$\clK\hra\clL$ be a finite separable field extension, 
and  $\clB\subset\clS$ be the integral closures of 
$\clA\subset\clR$ in the finite field extension $\clK\hra\clL$.
Then $\clB$ is finite $\clA$-module, and $\clS$ is a finite
$\clR$-module, in particular a Dedekind ring. 
\vskip2pt
Next let $\eur_1,\dots,\eur_r$ be the prime ideals 
of $\clB$ above $(\pi)$. Then each $\eur_i$ has 
height one, and the localizations $\clB_{\eur_i}$ are 
precisely the valuation rings of $\clL$ above the 
discrete valuation ring $\clA_{(\pi)}$ of $\clK$. And 
since $\clK\hra\clL$ is a finite separable extension, 
by the Finiteness Lemma, the fundamental equality 
holds:
\[
[\clL:\clK]=\sum_{i=1}^r e(\eur_i|\pi)\cdot f(\eur_i|\pi),
\]
where $e(\eur_i|\pi)$ and $f(\eur_i|\pi)=[\kappa(\eur_i):K]$ 
is the ramification index, respectively the residual degree 
of $\eur_i|\pi$. Therefore, if $v_\pi$ is the discrete valuation 
of $\clK$ with valuation ring $\clA_{(\pi)}$, one has
$K=\kappa(\pi)={\rm Quot}\big(\clA/(\pi)\big)=\kappa(v_\pi)$,
and the following are equivalent:
\vskip3pt
$\hhb2$i) There exists a prolongation $w$ of $v_\pi$ to
$\clL$ such that $[\clL:\clK]=[L:K]$, where $L:=\kappa(w)$.
\vskip3pt
ii) The ideal $\eur:=\pi\clB$ is a prime ideal of $\clB$, or
equivalently, $\pi$ is a prime element of $\clB$.
\vskip3pt
\noindent
In particular, if the above equivalent conditions i),~ii), hold,
then $\clB_\eur$ is the valuation ring of~$w$, and one has
$\kappa(\eur)={\rm Quot}\big(\clB/(\pi)\big)=\kappa(w)=L$. 
Further, $w$ is the unique prolongation of $v_\pi$ to $\clL$,
and $\eur=\pi\clB$ is the unique prime ideal of $\clB$ above
the ideal $\pi\clA$ of $\clA$.
\vskip2pt
We conclude by mentioning the following smoothness
criterion, which is a special case of the theory developed
in \nmnm{Kato}~\cite{Ka},~\S5; see also 
\nmnm{Green--Matignon}~\cite{GM1},~\S3, especially~3.4. 
\begin{fact} 
\label{KCR}
{\it In the above notations, suppose that 
$v_\pi$ has a prolongation $w$ to $\clL$ such that 
$[\clL:\clK]=[L:K]$, where $L:=\kappa(w)$ and 
$K=:\kappa(v_\pi)$. Let $A\hra B$ be the integral 
closure of $A=k[[t]]=\clA/(\pi)$ in the field extension $K\hra L$.
Let $\euD_{\clS|\clR}$ and $\euD_{B|A}$ be the differents
of the extensions of Dedekind rings $\clR\hra\clS$, 
respectively $A\hra B$. The following are equivalent:
\vskip2pt
$\hhb2${\rm i)}  $\Spec\clB$ is smooth over 
$\Spec\,\euo$.
\vskip2pt
{\rm ii)} The degrees of $\euD_{\clS|\clR}$ and 
$\euD_{B|A}$ are equal: $\deg(\euD_{\clS|\clR})=
\deg(\euD_{B|A})$.
\vskip3pt
If the above equivalent conditions are satisfied, then
there exists $Z\in\clB$ such that $\clB=\euo[[Z]]$ and
$B=\clS/(\pi)=k[[z]]$, where $z=Z\,\big({\rm mod}\,(\pi)\big)$.
\/}
\end{fact}
\section{The characteristic $p$ Oort Conjecture}
\begin{remark/definition}
\label{essential}
In the context of section~2), A), let $L|K$ be a cyclic 
extension of degree $p^\nx:=[L:K]$ with upper ramification 
jumps $\ux_1\leq \dots \leq \ux_\nx$. Recall that setting 
$\ux_0=0$, one has that $p\hhb1\ux_{\alrh-1}\leq\ux_{\alrh}$
for all~$\alrh$ with $0<\alrh\leq\nx$, and the inequality 
is strict if and only if $\ux_{\alrh}$ is not divisible by $p$. 
The division by $p$ gives:
\[
\ux_\alrh-p\hhb1\ux_{\alrh-1}=
                   p\hhb1q_{\alrh}+\epsilon_{\alrh}
\]
with $0\leq q_{\alrh}$ and $0\leq\epsilon_{\alrh}< p$,
and notice that by the remark above one has: 
$0<\epsilon_{\alrh}$ if and only if $(p,\ux_{\alrh})=1$ 
if and only if $p\hhb1\ux_{\alrh-1} < \ux_{\alrh}$.
\vskip2pt
We call $q_{\alrh}$ the \defi{essential part} of the upper 
jump at $\alrh$, and if $0< q_{\alrh}$ we say that 
$\ux_\alrh$ is an \defi{essential upper jump} for $L|K$, 
and that $\alrh$ is an \defi{essential upper index} for 
$L|K$. 
\vskip4pt
We introduce terminology as follows: Let $\clR\hra\clS$ 
be any generically finite Galois extension of Dedekind 
rings with cyclic inertia groups, and 
$\clK:={\rm Quot}(\clR)\hra{\rm Quot}(\clS)=:\clL$
be the corresponding cyclic extension of their quotient 
fields. For a maximal ideal  $\euq\in\Spec\clS$ above 
$\eup\in\Spec\clR$, let $\clK_\eup\hra\clL_\euq$ be the 
corresponding extension of complete discrete valued 
fields. We we will say that $\clR\hra\clS$ has \defi{(no) 
essential ramification jumps} at $\eup$, if the $p$-part 
of the cyclic extension of discrete complete valued fields 
$\clK_\eup\hra\clL_\euq$ has (no) essential upper 
ramification jumps. And we say that $\clR\hra\clS$ is has 
no essential ramification, if $\clR\hra\clS$ is has no 
essential ramification jumps at any $\eup\in\Spec\clR$.
\end{remark/definition}

In the remaining part of this subsection, we will work in
a special case of the situation presented in section~2,~C),
which is as follows: We consider a fixed algebraically 
closed field~$k$ with ${\rm char}(k)=p>0$, let 
$\euo=k[[\upi]]$ be the 
power series ring in the variable $\upi$ over $k$, thus
$\kk=k\lps\upi={\rm Quot}(\euo)$ is the Laurent power 
series in the variable $\upi$ over $k$. Let $\clA=k[[\upi,\TT]]$ 
and $\clK=k\lps{\upi,\TT}={\rm Quot}(\clA)$ be its field
of fractions. Then $A=\clA/(\upi)=k[[t]]$ 
and $K=k\lps t={\rm Quot}(A)$ is the fraction 
field of $A$. Further, $\clR:=\clR\otimes_\euo\kk
=\kk\langle\hhb{-1}\langle t\rangle\hhb{-1}\rangle$
is the ring of power series in $t$ over $\kk$ having 
$v_\kk$-bounded coefficients. Thus $\clR$ is a Dedekind 
ring having $\Spec(\clR)$ in bijection with the points
of the open rigid disc $\euX=\Spf\,\clR$ of radius~$1$ 
over the complete valued field~$\kk$. And we notice 
that $\euX=\Spec(\clR)$ is precisely the complement 
of $V(\upi)\subset\Spec(\clA)$. Finally, for a finite 
separable field extension $\clK\hra\clL$, we let 
$\clB\subset\clS$ be the integral closures of 
$\clA\subset\clR$ in the finite field extension $\clK\hra\clL$.
Thus $\clB$ is finite $\clA$-module, and $\clS$ is a finite
$\clR$-module, in particular a Dedekind ring. 
\begin{keylemma} 
\label{keylemma1}
{\rm (Characteristic $p$ local Oort conjecture)} \ \ 
In the above notations, let $N:=1+q_1+\dots+q_\nx$
and $x_1,\dots,x_N\in\eum_\euo$ be distinct points.
Let $K\hra L$ be a cyclic $\lvZ/p^\nx$-extension with 
upper ramification jumps $\ux_1\leq\dots\leq \ux_\nx$. 
Then there exists a cyclic $\lvZ/p^\nx$-extension 
$\clK\hra\clL$ such that the integral closure $\clA\hra\clB$ 
of $\clA$ in $\clK\hra\clL$ and the corresponding 
extension of Dedekind rings $\clR\hra\clS$ satisfy:
\vskip3pt
{\rm 1)} The morphism $\varphi:\Spec\clB\to\Spec\euo$ 
is smooth. In particular, $\clB=k[[\upi,Z]]$ and the 
special fiber of $\varphi$ is $\Spec B\to k$, where 
$B=k[[z]]=\clB/(\upi)$ and $z=Z\,\big({\rm mod}\,(\upi)\big)$.
\vskip3pt
{\rm 2)} The canonical morphism $\clR\hra\clS$ has 
no essential ramification and is ramified only at points 
$y_\mu\in\Spec\clS$ above the points $x_\mu\in\Spec\clR$, 
$1\leq\mu\leq N$.
\vskip3pt
{\rm3)}  Let $(\ux_{\mu,\alrh})_{1\leq\alrh\leq\nx_\mu}$ 
be the upper ramification jumps at each $y_\mu\mapsto x_\mu$.
Then $(\nx_\mu)_{1\leq\mu\leq N}$ is decreasing, and the
upper jumps are given by:
\begin{itemize}
\vskip2pt
\item[{\rm i)}] $\ux_{1,\rhal}=p\hhb1\ux_{1,\rhal-1}+\epsilon_{\rhal}$
for $\,1\leq\rhal\leq\nx$.
\vskip1pt\noindent \ \
\item[{\rm ii)}] $\ux_{\mu,\rhal}=p\hhb1\ux_{\mu,\rhal-1}+p-1$
for $1 < \mu\leq N$ and $1\leq\alrh\leq \nx_\mu$.
\end{itemize}

\vskip2pt
{\rm4)} In particular, the branch locus $\,\{x_1,\dots,x_N\}$ 
of $\,\clR\hra\clS$ is independent of $k[[t]]\hra k[[z]]$, and 
the upper ramification jumps 
$\uix_\mu:=(\ux_{\mu,1},\dots,\ux_{\mu,\nx_\mu})$ at 
each $y_\mu\mapsto x_\mu$, $1\leq\mu\leq N\!$, depend 
only on the upper jumps 
$\uix:=(\ix_1,\dots,\ix_\nx)$ of $k[[t]]\hra k[[z]]$.
\end{keylemma}
The proof of the Key Lemma~\ref{keylemma1} will 
take almost the whole section. We begin by recalling that 
in the notations from section~2,~B), there exists 
${\undr p}=\big(p_1(\tmu) ,\dots,p_\nx(\tmu) \big)$, say 
in standard form, such that $L=K_{\undr p}$. The integral
closure $A\hra B$ of $A=k[[t]]$ in the field extension
$K\hra L$ is of the form $B=k[[z]]$ for any uniformizing
parameter $z$ of $L={\rm Quot}(B)$. And the degree of the 
different $\euD_{L|K}:=\euD_{B|A}$ is $\,\deg(\euD_{L|K})=
\sum_{\alrh=1}^\nx(\ux_\alrh+1)(p^{\hhb1\alrh}-p^{\alrh-1})$.
%
%
\vskip7pt
\noindent
A) {\it Combinatorics of the upper jumps\/}
\vskip7pt
\vskip2pt
In the above context, let $\nx_0$ be the number 
of essential upper jumps, which could be zero. If 
there exist essential upper jumps, i.e., $0 < \nx_0$, 
let $r_1\leq\dots\leq r_{\nx_0}$ be the essential upper 
indices for $L|K$, and notice that the sequence 
$(r_i)_{1\leq i\leq\nx_0}$ is strictly increasing with 
$r_{\nx_0}\leq \nx$. For technical reasons (to simplify
notations) we set $r_{\nx_0+1}:=\nx+1$, and if we 
need to speak about $r_{\nx_0+1}$, we call
it the \defi{improper upper index}, which for $\nx_0=0$
would become $r_1=e+1$.
\vskip2pt
We next construct a finite strictly increasing sequence 
$(d_i)_{0\leq i\leq\nx_0}$ as follows: We set $d_0=1$,
and we are done if $\nx_0=0$. If $\nx_0>0$, we define 
inductively $d_i:=d_{i-1}+q_{r_i}$ for $1\leq i\leq\nx_0$. 
In particular, we see that $N:=1+q_1+\dots q_\nx$ is 
$N=1$ if $\nx_0=0$, and $N:=d_{\nx_0}$ otherwise.
\vskip2pt
We define an $N\times\nx$ matrix of non-negative integers 
$(\tht_{\mu,\rhal})_{1\leq\mu\leq N,\,1\leq\rhal\leq\nx}$ 
as follows: 
\vskip5pt
$\bullet$ If $\nx_0=0$, then $N=1$, and the $1\times\nx$ 
matrix is given by $\tht_{1,\rhal}:=u_\rhal$, $1\leq\rhal\leq\nx$.
\vskip2pt
$\bullet$ If $\nx_0>0$, thus $N>1$, we 
define:\footnote{Here and elsewhere, we set 
$\tht_{\mu,0}:=0$ as well as $\ux_0=0$ and $\ux_{\mu,0}=0$.}
\begin{itemize}
\item[{\rm a)}] $\tht_{1,\rhal}=p\hhb1\tht_{1,\rhal-1}+\epsilon_{\rhal}$
for $\,1\leq\rhal\leq\nx$.
\vskip5pt\noindent \ \
\item[{\rm b)}] For $i=1,\dots,\nx_0$ \ and \ $d_{i-1}<\mu\leq d_i$ define:
\vskip2pt\noindent
$\scriptstyle\bullet$ \ $\tht_{\mu,\rhal}=0$ for $1\leq\rhal < r_i$.
\vskip2pt\noindent
$\scriptstyle\bullet$ \ $\tht_{\mu,\rhal}=p\hhb1\tht_{\mu,\rhal-1}+p-1$ 
for $r_i\leq\rhal\leq\nx$.
\end{itemize}

\vskip3pt
Notice that in the case $\nx_0>0$, one has: Let
$\rhal$ with $1\leq\rhal\leq\nx$ be given. Consider the
unique $1\leq i\leq\nx_0$ such that $r_i\leq\rhal< r_{i+1}$.
(Recall the if $r_i=\nx$, then $r_{i+1}:=\nx+1$ by
the convention above!) Then for all $\mu$ with 
$1\leq\mu\leq N$ one has: $\tht_{\mu,\rhal}\neq0$ 
if and only if $\mu\leq d_i$. 
The fundamental combinatorial property of
$(\tht_{\mu,\rhal})_{1\leq\mu\leq N,\,1\leq\rhal\leq\nx}$ 
is given by the following:  
\begin{lemma} 
\label{combinlemma}
For $1\leq i\leq \nx_0$ and $r_i\leq\rhal< r_{i+1}$ one has: 
$\ux_\rhal+1=\sum_{1\leq\mu\leq d_i}(\tht_{\mu,\rhal}+1)$.
\end{lemma}
\begin{proof}
The proof follows by induction on $\rhal=1,\dots,\nx$. 
Indeed, if $\nx_0=0$, then $N=1$, and there is nothing
to prove. Thus supposing that $\nx_0>0$, one argues as
follows: 
\vskip2pt
$\bullet$ The assertion holds for $\rhal=1$: First, if $r_1>1$, 
then $\tht_{\mu,1}=0$ for $1<\mu$, thus there is nothing to
prove. Second, if $r_1=1$, then $q_1>0$ and $d_1=1+q_1$.
Further, by the definitions one has: $\tht_{1,1}=\epsilon_1$
and $\tht_{\mu,1}=p-1$ for $1<\mu\leq d_1$, and conclude
by the fact that $u_1=p\hhb1q_1+\epsilon_1$.
\vskip2pt
$\bullet$ If the assertion of Lemma~\ref{combinlemma} 
holds for $\rhal<\nx$, the assertion also holds for $\rhal+1$: 
Indeed, let $i$ be such that $r_i\leq\rhal< r_{i+1}$. 
\vskip3pt
\underbar{Case 1}: $\rhal+1< r_{i+1}$. \ Then 
$r_i\leq\rhal+1<r_{i+1}$, and in particular, $\rhal+1$ is 
not an essential jump index. Hence by definitions one has 
that $\ux_{\rhal+1}=p\hhb1\ux_\rhal+\epsilon_{\rhal+1}$ 
with $0\leq\epsilon_{\rhal+1}< p$. On the other hand, by 
the induction hypothesis we have that
$\ux_\rhal=\tht_{1,\rhal}+\sum_{1<\mu\leq d_i}(\tht_{\mu,\rhal}+1)$.
Hence taking into account the definitions of $\tht_{\mu,\rhal}$, 
we conclude the proof in Case~1 as follows:
\begin{eqnarray*} 
\ux_{\rhal+1}+1
    &=&p\hhb1\ux_\rhal+\epsilon_{\rhal+1}+1\\
   &=&p\hhb1\tht_{1,\rhal}+
              \textstyle\sum_{1<\mu\leq d_i}
                  (\hhb1p\hhb1\tht_{\mu,\rhal}+p)+\epsilon_{\rhal+1}+1\\
   &=&(p\hhb1\tht_{1,\rhal+1}+\epsilon_{\rhal+1}+1)+
             \textstyle\sum_{1 < \mu\leq d_i}
              \big( (\hhb1p\hhb1\tht_{\mu,\rhal}+p-1)+1\big)\\
    &=&(\ux_{1,\rhal+1}+1)+ \textstyle\sum_{1 < \mu\leq d_i}
              (\tht_{\mu,\rhal+1}+1)\\
    &=&\textstyle\sum_{1\leq\mu\leq d_i}
              (\tht_{\mu,\rhal+1}+1).
\end{eqnarray*}

\underbar{Case~2}: $\rhal +1 = r_{i+1}$. Then $\rhal+1$ 
is an essential jump index, thus by definitions one has:
$\ux_{\rhal+1}=p\hhb1\ux_\rhal+p\hhb1q_{\rhal +1}+\epsilon_{\rhal+1}$
with $0< q_{\rhal+1}$ and $0 < \epsilon_{\rhal+1} < p$, 
$d_{i+1}=d_i+q_{\rhal+1}$, $r_{i+1}\leq\rhal +1< r_{i+2}$.
On the other hand, by the induction hypothesis one has 
$\ux_\rhal=\tht_{1,\rhal}+
                  \sum_{1< \mu\leq d_i}(\tht_{\mu,\rhal}+1)$.
Therefore, using the definitions of $\tht_{\mu,\rhal}$ we get:
\begin{eqnarray*} 
\ux_{\rhal+1} +1
    &=&p\hhb1\ux_\rhal+p\hhb1q_{\rhal +1}+\epsilon_{\rhal+1}+1\\
   &=&(p\hhb1\tht_{1,\rhal}+\epsilon_{\rhal+1}+1)+
              \textstyle\sum_{1<\mu\leq d_i}
                  (\hhb1p\hhb1\tht_{\mu,\rhal}+p)
                        +p\hhb1q_{\rhal +1}\\
   &=&(\ux_{1,\rhal+1}+1)+
             \textstyle\sum_{1 < \mu\leq d_i}
              \big( (\hhb1p\hhb1\tht_{\mu,\rhal}+p-1)+1\big)+
 \textstyle\sum_{d_i < \mu\leq d_{i+1}}\big((p-1)+1\big)\\
    &=&(\ux_{1,\rhal+1}+1)+ \textstyle\sum_{1 < \mu\leq d_{i+1}}
              (\tht_{\mu,\rhal+1}+1)\\
    &=&\textstyle\sum_{1\leq\mu\leq d_{i+1}}
              (\tht_{\mu,\rhal+1}+1).
\end{eqnarray*}

This completes the proof of Lemma~\ref{combinlemma}.
\end{proof}

\vskip5pt
\noindent
B) {\it Generic liftings\/}
\vskip7pt
Let $k[[t]]\hra k[[z]]$ be a $\lvZ/p^\nx$-cyclic extension
with upper ramification jumps $\ux_1\leq\dots\leq\ux_\nx$. 
Recall that setting $K=k\lps t$ and $L=k\lps z$, in the 
notations introduced in section~2,~B), there exists 
${\undr p}=\big(p_1(\tmu) ,\dots,p_\nx(\tmu) \big)$ such that
$L=K_{\undr p}$, where ${\undr p}$ is in standard form, i.e., 
either $p_\alrh(\tmu)=0$ or it contains no non-zero terms in 
which the exponent of $\tmu$ is divisible by~$p$. And by the 
discussion in~section~2,~B), one has that
\[
\ux_\alrh=\max\,\{\hhb2p\hhb1\ux_{\alrh-1},\dgr{p_\alrh(\tmu)}\},
\quad\alrh=1,\dots,\nx.
\]
\begin{fact}
\label{fact2}
Let ${\undr c}:=\big(h_1(\tmu),\dots,h_\nx(\tmu)\big)$ be an 
arbitrary Witt vector with coordinates in~$k[\tmu]$. 
For a fixed polynomial $h(\tmu)\in k[\tmu]$, let ${\undr c}_{h,i}$ 
be the Witt vector whose $i^{\rm th}$ coordinate 
is $h(\tmu)$, and all the other coordinates are equal 
to~$0\in k[\tmu]$. Then $\tilde {\undr c}:={\undr c}+{\undr c}_{h,i}$ 
has coordinates $\tilde{\undr c}=\big(\tilde h_1(\tmu),\dots,
\tilde h_\nx(\tmu)\big)$ satisfying the following:
\begin{itemize}
\item[a)] $\tilde h_j(\tmu)=h_j(\tmu)$ for $j < i$.
\vskip2pt
\item[b)] $\tilde h_i(\tmu)=h_i(\tmu)+h(\tmu)$.
\vskip2pt
\item[c)] $\dgr{\tilde h_j(\tmu)}\leq
      \max\,\{\hhb2\dgr{h_j(\tmu)},\,p^{j-i}\dgr{h(\tmu)}\}$
      for all $i < j$.
\end{itemize}      
\end{fact}
\vfill\eject
\begin{definition/remark} 
\label{normaliz}
$\hhb1$
\vskip2pt
1) In the above context, let 
${\undr q}:=\big(q_1(\tmu),\dots,q_\nx(\tmu)\big)$ with 
$q_\alrh(\tmu)\in k[\tmu]$ be some generator of $L|K$, i.e., 
$L=K_{\undr q}$. We say that ${\undr q}$ is \defi{normalized}
if $\ux_\alrh=\dgr{q_\alrh(\tmu)}$, $\alrh=1,\dots,\nx$. And
we say that ${\undr q}$ is separable, if each $q_\alrh(\tmu)$
is a separable polynomial (in $\tmu$).
\vskip2pt
2) We notice that if 
${\undr q}:=\big(q_1(\tmu),\dots,q_\nx(\tmu)\big)$
is some given Witt vector and $L:=K_{\undr q}$, then 
${\undr q}$ is normalized if and only if it satisfies: 
$\dgr{q_1(\tmu)}$ is prime to~$p$, and for all $1\leq\alrh < \nx$ 
one has that $p\hhb1|\hhb1\dgr{q_{\alrh+1}(\tmu)}$ 
implies $\dgr{q_{\alrh+1}(\tmu)}=p\,\dgr{q_\alrh(\tmu)}$.
\vskip2pt
3) Given a generator 
${\undr p}=\big(p_1(\tmu),\dots,p_\nx(\tmu)\big)$
in standard form for $L|K$, one can construct 
a separable normalized generator 
${\undr q}=\big(q_1(\tmu),\dots,q_\nx(\tmu)\big)$ 
as follows: Consider Witt vectors of the form 
${\undr c}:=\big(h_1(\tmu),\dots,h_\nx(\tmu)\big)$ 
with $h_\alrh(\tmu)\in k[\tmu]$ and 
$\dgr{h_\alrh(\tmu)}=\ux_{\alrh-1}$ for $1< \alrh\leq \nx$,
which are ``inductively generic'' with those properties.
Then setting 
\[
{\undr q}=:{\undr p}+\wp_\nx({\undr c})=:
                     \big(q_1(\tmu),\dots,q_\nx(\tmu)\big),
\]
it follows that $K_{\undr q}=L=K_{\undr p}$, thus ${\undr q}$ 
is a representative for ${\undr p}$ modulo $\wp_\nx(K)$. 
And applying inductively Fact~\ref{fact2} above, 
one gets: If $p\hhb1 \ux_{\alrh-1}<\ux_\alrh$, then 
$\dgr{h_\alrh(\tmu)}=\dgr{p_\alrh(\tmu)}=\ux_\alrh$. 
Second, if $p\hhb1 \ux_{\alrh-1}=\ux_\alrh$, then 
$\dgr{p_\alrh(\tmu)} < \ux_\alrh=p\hhb1\ux_{\alrh-1}$. Hence 
applying Fact~\ref{fact2} inductively, since $h_\alrh(\tmu)$ 
is generic of degree $\ux_{\alrh-1}$,  we get:
$\dgr{q_\alrh(\tmu)}=p\hhb1\dgr{h_{\alrh-1}(\tmu)}
            =p\hhb1\ux_{\alrh-1}=\ux_\alrh$, and
each $q_\alrh(\tmu)$ is separable.
\end{definition/remark}
Coming back to our general context, let
${\undr p}=\big(p_1(\tmu),\dots,p_\nx(\tmu)\big)$ 
be a {\it normalized generator\/} for $L|K$. Let 
$(\tht_{\mu,\rhal})_{1\leq\mu\leq N,\,1\leq\rhal\leq\nx}$
be the matrix of non-negative integers produced 
in the previous subsection~A). Since $k$ is algebraically 
closed, we can write each polynomial $p_\alrh(\tmu)$ 
as a product of polynomials $p_{\mu,\alrh}(\tmu)$ as follows: 
\[
p_\alrh(\tmu)=\textstyle{\prod}_{1\leq\mu\leq N}
             \,p_{\mu,\alrh}(\tmu),
\]
with $\dgr{p_{1,\alrh}(\tmu)}=\tht_{1,\alrh}$, 
$\dgr{p_{\mu,\alrh}(\tmu)}=\tht_{\mu,\alrh}+1$ for
$\mu >1, \tht_{\mu,\alrh}\neq0$,
$p_{\mu,\alrh}=1$ if $\tht_{\mu,\alrh}=0$.
\vskip5pt
For the given elements $x_\mu\in\upi\euo$, $\mu=1,\dots,N$
we set $\TT_\mu:=\TT-x_\mu\in\euo[\TT]$. And for a 
fixed choice $p_{\mu,\alrh}(\tmu)\in k[\tmu]$, we let 
$P_{\mu,\alrh}(\TTmu)\in\euo[\TTmu]$ 
be ``generic'' preimages with 
$\dgr{P_{\mu,\alrh}(\TTmu)}=
               \dgr{p_{\mu,\alrh}(\tmu)}$ for $\mu=1,\dots,N$
and $\alrh=1,\dots,\nx$. In particular, each
\[
P_\alrh:=\textstyle\prod_\mu\,P_{\mu,\alrh}(\TTmu)
 \in\euo[\TT^{-1}_{x_1},\dots,\TT^{-1}_{x_N}]\subset
    \clA_{x_1,\dots,x_N}, \quad \alrh=1,\dots,\nx
\]
is a linear combination of monomials in 
$(\TT-x_1)^{-1},\dots,(\TT-x_N)^{-1}$ with ``general'' 
coefficients from $\euo$ such that under the specialization 
homomorphism $\clA_{x_1,\dots,x_N}\to A[\tmu]$ they 
map to $p_\alrh=p_\alrh(\tmu)$. To indicate this, we will write 
for short 
\[
P_{\mu,\alrh}(\TTmu)\mapsto p_{\mu,\alrh}(\tmu),
\quad P_\alrh\mapsto p_\alrh(\tmu).
\]
\vskip4pt
We set ${\undr P}:=(P_1,\dots,P_\nx)$ and view it as a Witt vector
of length $\nx$ over $\clK$, and consider the corresponding
cyclic field extension $\clL:=\clK_{\undr P}$. Let $\clA\hra\clB$
be the normalization of $\clA$ in~$\clK\hra\clL$. Since
$\clA=k[[\upi,t]]$ is Noetherian and $\clK\hra\clL$ is separable,
it follows that $\clB$ is a finite $\clA$-algebra, thus Noetherian.
And since $\clA$ is local and complete, so is $\clB$.
\vskip4pt
We next have a closer look at the branching in the
finite ring extension $\clA\hra\clB$. For that we view 
$\clA\hra\clB$ as a finite morphism of $\euo:=k[[\upi]]$
algebras, and introduce geometric language as
follows: $\clX=\Spec\clA$ and $\clY=\Spec\clB$. 
Thus $\clA\hra\clB$ defines a finite $\euo$-morphism 
$\clY\to\clX$. Further let $\euY:=\Spec\clS\to\Spec\clR=:\euX$ 
and $Y:=\Spec\clB/(\upi)\to\Spec\clA/(\upi)=:X$ be
the generic fiber, respectively the special fiber
of $\clY\to\clX$. In particular, $X=\Spec A$ and $Y\to X$ 
is a finite morphism. We further mention the following 
general fact for later use:
\begin{fact}
\label{fact3}
Let $\clK\hra\clL$ be a cyclic extension of degree
$[\clL:\clK]=p^\nx$ with Galois group $G=\lvZ/p^\nx$,
say defined by some Witt vector 
${\undr a}=(a_1,\dots,a_\nx)\in\Wnx(\clK)$. 
For every $0\leq m\leq \nx$ let $\clK\hra\clL_m$ 
be the unique sub-extension of $\clK\hra\clL$ with 
$[\clL_m:\clK]=p^m$, hence $\clL_0=\clK$ and
$\clL_\nx=\clL$. Let $v$ be a discrete valuation of $\clK$, 
say with valuation ring $\clO_v\subset\clK$ and 
residue field $\clO_v\to\kp v$, $f\mapsto\oli f$, and 
let $T_v\subseteq Z_v\subseteq G$ be the inertia,
respectively decomposition, subgroups of $v$ in
$G$. Then for all~$m$ with $1\leq m\leq \nx$ the 
following hold:
\vskip2pt
\begin{itemize}
\item[1)] Let $v(a_1),\dots,v(a_m)\geq0$. Then 
$T_v\subseteq p^mG$, and $Z_v\subset p^mG$ 
iff $(\oli a_1,\dots,\oli a_m)\in{\rm im}(\wp_m)$.
\vskip3pt
\item[2)] If $v(a_m)$ is negative and prime to $p$, 
then $p^{m-1}G\subseteq T_v$.
\end{itemize}
\end{fact}  
\vskip2pt
We notice that since $\clA=k[[\upi,t]]$ is a two dimensional 
local regular ring, $\clX$ is a two dimensional regular 
scheme. Therefore, the branch locus of $\clY\to\clX$ is 
of pure co-dimension one. Thus in order to describe the
branching behavior of $\clY\to\clX$ one has to describe
the branching at the generic point $(\upi)$ of the special 
fiber $X\subset\clX$ of $\clX$, and at the closed points 
$x$ of the generic fiber $\euX\subset\clX$ of $\clX$.
%
\vskip5pt
\noindent
$\bullet$ {\it The branching at $(\upi)$\/}
\vskip4pt
We recall that $\clK\hra\clL$ is defined as a cyclic
extension by ${\undr P}:=(P_1,\dots,P_\nx)$, where 
each $P_\alrh$ is of the form $P_\alrh=\prod_\mu 
P_{\mu,\alrh}(\TTmu)$ with $P_{\mu,\alrh}
(\TTmu)\in\euo[\TTmu]\subset\clA_{x_1,\dots,x_N}$ 
is some generic preimage 
of $p_{\mu,\alrh}(\tmu)$ with degree satisfying  
$\dgr{P_{\mu,\alrh}(\TTmu)}=\dgr{p_{\mu,\alrh}(\tmu)}$.
In particular, the elements $P_{\mu,\alrh}(\TTmu)$
are not divisible by $\upi$ in the factorial ring 
$\clA_{x_1,\dots,x_N}$. Therefore, the elements 
$P_{\mu,\alrh}(\TTmu)$ are units $\clA_{(\upi)}$.
By Fact~\ref{fact3} we conclude that $\upi$ is not 
branched in $\clK\hra\clL$, and therefore, the 
special fiber $Y\to X$ of $\clY\to\clX$ is reduced. 
Moreover, since $P_1\mapsto p_1(\tmu)$, and 
the latter satisfies $p_1(\tmu)\not\in\wp(K)$, it 
follows by~Fact~\ref{fact3},~1), that ${\rm Gal}(\clL|\clK)$ 
is contained in the decomposition group of $v_{\upi}$.
In other words, $\upi$ is totally inert in 
$\clK\hra\clL$. In particular, $\clY\to\clX$ is \'etale 
above $\upi$, and moreover, the special fiber $Y\to X$ 
of $\clY\to\clX$ is reduced, irreducible, and generically
cyclic Galois of degree $p^\nx=[L:K]$.
\vskip5pt
\noindent
$\bullet$ {\it The branching at the points of the generic
fiber $x\in\euX$\/}
\vskip4pt
Recall that $\kk=k\lps\upi$ and that $\euX=\Spec\clR$, 
where $\clR=\clA\otimes_{k[[\upi]]}\hhb{-1}\kk$ is the 
ring of power series in $t$ with bounded coefficients 
from the complete discrete valued field $\kk$. 
[Thus $\euX$ is actually the rigid open unit disc 
over $\kk$.] If $x\in\euX$ is a closed point 
different from $x_1,\dots,x_N$, and $\clA_x$ 
is the local ring of $\euX$ at $x$, it follows 
that $P_1,\dots,P_\nx \in \clA_\eup$. Hence by 
Fact~\ref{fact3},~1), it follows that $x$ is not branched 
in $\clK\hra\clL$. Thus it is left to analyze the branching behavior of 
$\euY\to\euX$ at the closed points $x_1,\dots,x_N\in\euX$.
In this process we will also compute the total contribution 
of the ramification above $x_\mu$ to the total 
different~$\euD_{\clS|\clR}$ for $\mu=1,\dots,N$. 
\vskip3pt
Recall that every $x_\mu$ is a $\kk$ rational point 
of $\euX$, and $\TT_\mu:=\TT-x_\mu$ is the ``canonical'' 
uniformizing parameter at $x_\mu$. Thus 
$\clK_\mu:=\kk\lps{\TT_\mu}$ is the quotient field 
of the completion of the local ring at $x_\mu$, and 
we denote by $v_\mu:\clK_\mu^\times\to\lvZ$ the 
canonical valuation at $x_\mu$. We notice that 
$\TT_{\nu}=x_\mu-x_\nu +\TT_{\mu}$, hence by 
the ``genericity'' of $P_{\mu,\alrh}(\TTmu)$,
we can and \underbar{will} suppose that 
$P_{\nu,\alrh}(\TT^{-1}_{\nu})$ is a $v_\mu$-unit 
in $\clK_\mu$. We conclude that there exist 
$v_\mu$-units $\eta_1,\dots\eta_\nx \in \clK_\mu$
such that denoting by $\clL_\mu$ the compositum 
of $\clK_\mu$ and $\clL$, the cyclic extension 
$\clK_\mu\hra\clL_\mu$ is defined by the Witt vector:
\vskip3pt
\centerline{${\undr P}_\mu=\big(\eta_1P_{\mu,1}(\TTmu),
       \dots,\eta_\nx P_{\mu,\nx}(\TTmu)\big)$.}
\vskip5pt
\indent
\underbar{Case 1}: $\mu=1$. \ 
\vskip3pt
\noindent
First, by definitions we have 
$\dgr{P_{1,\alrh}(\TT^{-1}_{1})}=
               \dgr{p_{1,\alrh}(\tmu)}=\tht_{1,\alrh}$ 
for~all~$\alrh$. Second, by the definitions of 
$(\tht_{1,\alrh})_{1\leq\alrh\leq\nx}$ it follows that 
$\big(p_{1,1}(\tmu),\dots,p_{1,\nx}(\tmu)\big)$ is 
actually a normalized system of polynomials 
in $k[\tmu]$. Hence by Definition/Remark~\ref{normaliz},~3),
it follows that the upper ramification jumps of 
$\clK_1\hra\clL_1$ are precisely 
$\tht_{1,1}\leq\dots\leq\tht_{1,\nx}$, and in particular,
one has $[\clL_1:\clK_1]=p^\nx$.
\vskip5pt
\underbar{Case 2}: $1< \mu$, hence one has $0 < N_0$ as well.
\vskip3pt
\noindent
In the notations from the previous subsection~A), let 
$1\leq i \leq N_0$ maximal be such that $d_{i-1}<\mu$. Then 
by the definition of  $(\tht_{\mu,\alrh})_{1\leq\alrh\leq\nx}$ 
we have: $\tht_{\mu,\alrh}=0$ for $\alrh< r_i$, 
$\tht_{\mu,r_i}=p-1$, and $\tht_{\mu,\alrh}=
                                     p\hhb1\tht_{\mu,\alrh-1}+p-1$ 
for $r_i\leq\alrh\leq\nx$. Further, again by definitions,
one has $P_{\mu,\alrh}(\TTmu)=1$ for 
$\tht_{\mu,\alrh}=0$, i.e., for $\alrh < r_i$. And 
$P_{\mu,\alrh}(\TTmu)$ is generic of degree 
$\tht_{\mu,\alrh}+1$ for $r_i\leq\alrh\leq \nx$. Hence
the Witt vector 
${\undr P}_\mu:=\big(\eta_1P_{\mu,1}(\TTmu),
       \dots,\eta_\nx P_{\mu,\nx}(\TTmu)\big)$
satisfies the following conditions: 
\vskip2pt
- \ $\eta_\alrh P_{\mu,\rhal}(\TTmu)=\eta_\alrh$
for $1\leq\alrh< r_i$.
\vskip2pt
- \ $\eta_\alrh P_{\mu,\rhal}(\TTmu)$ is a 
generic polynomial of degree $\tht_{\mu,\alrh}+1$ 
for $r_i\leq\alrh\leq\nx$.
\vskip3pt
\noindent
In particular, by Fact~\ref{fact3}, one has that
$v_\mu$ is unramified in the sub-extension 
$\clK\hra\clL_{r_i-1}$ of degree $p^{r_i-1}$ of 
$\clK\hra\clL$. We claim that $\clL_{r_i-1}\hra\clL$ 
is actually totally ramified, or equivalently, that 
${\rm Gal}(\clL|\clL_{r_i-1})\subseteq {\rm Gal}(\clL|\clK)$
is the inertia group of $v_\mu$, hence the ramification
subgroup of $v_\mu$, because there is no tame
ramification involved. Indeed, it is sufficient to 
prove that this is the case after base changing 
everything to the maximal unramified extension 
$\clK_\mu\hra\tlclK_\mu$ of $\clK_\mu$. Recall that for 
$\rhal$ with $r_i\leq\alrh\leq \nx$ we have by definitions that 
$P_{\mu,\rhal}(\TTmu)\in\euo[\TTmu]$ 
is a generic polynomial in $\TTmu$ over 
$\euo=k[[\upi]]$ of degree $\tht_{\mu,\rhal}+1$. Hence
for $\rhal=r_i$ we have: $\tht_{\mu,r_1}+1$ is
divisible by $p$ and $P_{\mu,r_i}(\TTmu)$ 
is generic. But then it follows that the standard 
representative $Q_{\mu,r_i}(\TTmu)
\in\scl\kk[\TTmu]$ of $P_{\mu,\rhal}(\TTmu)$ 
modulo $\wp(\tlclK_\mu)$ has degree $\tht_{\mu,\rhal}$. 
Thus by Fact~\ref{fact3},~3), it follows that $v_\mu$ is totally 
ramified in the field extension $\tlclK_\mu\hra\tlclL_\mu$ 
and that $p^{\nx-r_i+1}=[\tlclL_\mu:\tlclK_\mu]$. Combining 
this with the fact that $v_\mu$ is unramified in 
$\clK\hra\clL_{r_i-1}$ and $p^{r_i-1}=[\clL_{r_i-1}:\clK]$, 
it follows that $\clK\hra\clL_{r_i-1}$ is the ramification 
field of $v_\mu$ in $\clK\hra\clL$. Equivalently, 
${\rm Gal}(\clL|\clL_{r_i-1})\subseteq{\rm Gal}(\clL|\clK)$ 
is the ramification subgroup of $v_\mu$.
\vskip2pt
We next compute the degree of the local different
of $\clR\hra\clS$ above $x_\mu$. Recall that
$\clK\hra\clL$ is defined by the Witt vector 
\[
{\undr P}_\mu=\big(\eta_1P_{\mu,1}(\TTmu),
       \dots,\eta_\nx P_{\mu,\nx}(\TTmu)\big).
\]
On the other hand, ${\undr P}_\mu$ is equivalent modulo
$\wp_\nx(\tlclK)$ to its standard form 
\[
{\undr P}'_\mu=\big(Q_{\mu,1}(\TTmu),
         \dots,Q_{\mu,\nx}(\TTmu)\big)
\]
with $Q_{\mu,\alrh}=0$ for $1\leq\alrh< r_i$
and $\dgr{Q_{\mu,\nx}(\TTmu)}=\tht_{\mu,\alrh}$
for $r_i\leq\alrh\leq \nx$. Hence setting 
\[
{\undr Q}_\mu=\big(Q_{\mu,r_i}(\TTmu),
        \dots,Q_{\mu,\nx}(\TTmu)\big),
\]
it follows that ${\undr Q}_\mu$ is a Witt vector of length 
$\nx-r_i+1$, and ${\undr Q}_\mu$ is in standard form, and 
the cyclic extension $\tlclK_\mu\hra\tlclL_\mu$ is defined
by ${\undr Q}_\mu$. Hence setting $\nx_\mu:=\nx-r_i+1$,
it follows that the cyclic field extension 
$\tlclK_\mu\hra\tlclL_\mu$ has degree $p^{\nx_\mu}$ 
and upper ramification jumps given by
\vskip9pt
\centerline{\hhb{20}$(*)$\hhb{60}
$\ux_{\mu,\alp}=\dgr{Q_{\mu,r_i+\alp}(\TTmu)}=
     \tht_{\mu,\hhb1r_i+\alp-1},\quad \alp=1,\dots, \nx_\mu$.\hhb{80}}
\vskip15pt
\noindent
\vfill\eject
\noindent
C) \  {\it Finishing the proof of Key Lemma~\ref{keylemma1}\/}
\vskip10pt
Let $\euD_\mu$ be the local part above $x_\mu$ of 
the global different $\euD_{\!\clS|\clR}$ of the 
extension of Dedekind rings $\clR\hra\clS$. Then 
if $\clK\hra\clL^Z\hra\clL^T\hra\clL$ are the 
decomposition/inertia subfields of $v_\mu$ in the
cyclic field extension $\clK\hra\clL$, by the functorial
behavior of the different, it follows that
\[
\deg(\euD_\mu)=[\kappa(x_\mu):\kk]\cdot[\clL^T:\clK]\cdot
     \deg(\euD_{\!\tlclL_\mu|\tlclK_\mu}).
\]
On the other hand, since $x_\mu\in\eum_\euo$, one
has $\kappa(x_\mu)=\kk$. Further, by the discussion
above one has that $\clL^T=\clL_{r_i-1}$, 
thus $p^{r_i-1}=[\clL^T:\clK]$. And
$\deg(\euD_{\!\tlclL_\mu|\tlclK_\mu})$ can be 
computed in terms of upper ramification jumps as 
indicated at the end of~section~1),~A):
\[
\deg(\euD_{\!\tlclL_\mu|\tlclK_\mu})=
    \textstyle\sum_{1\leq\alp\leq\nx_\mu}
          (\ux_{\mu,\alp}+1)(p^\alp-p^{\alp-1}).
\]
Hence taking into account the discussion above, we get:     
\begin{eqnarray*}
\deg(\euD_\mu)
  &=&[\kappa(x_\mu):\kk]\cdot[\clL^T:\clK]\cdot
     \deg(\euD_{\!\tlclL_\mu|\tlclK_\mu})\\
  &=&p^{r_i-1}\textstyle\sum_{1\leq\alp\leq\nx_\mu}
     (\ux_{\mu,\alp}+1)(p^\alp-p^{\alp-1})\\
  &=&\textstyle\sum_{1\leq\alp\leq\nx_\mu}
     (\tht_{\mu,r_i+\alp-1}+1)(p^{\alp+r_i-1}-p^{(\alp+r_i-1)-1})\\
  &=&\textstyle\sum_{r_i\leq\rhal\leq\nx}(\tht_{\mu,\rhal}+1)
          (p^{\rhal}-p^{\rhal-1})
\end{eqnarray*}
\vskip5pt
\noindent
Recall that $\ux_\rhal+1=\sum'_\mu(\tht_{\mu,\rhal}+1)$
for all $1\leq\alrh\leq \nx$, where $\sum'_\mu$ is taken 
over all $\mu$ with $\tht_{\mu,\rhal}\neq0$. Further, 
$\deg(\euD_1)=\sum_{1\leq\rhal\leq\nx}
                                  (\tht_{1,\rhal}+1)(p^\rhal-p^{\rhal-1})$ 
and $\deg(\euD_\mu)=\sum_{r_i\leq\rhal\leq\nx}
                               (\tht_{\mu,\rhal}+1)(p^\rhal-p^{\rhal-1})$ 
for all $1<\mu\leq N$. Therefore we get the following:
\begin{eqnarray*}
\deg(\euD_{\clS|\clR})
  &=&\micsm{1\leq\mu\leq N}{}\deg(\euD_\mu)\\
  &=&\micsm{1\leq\rhal\leq\nx}{}(\tht_{1,\rhal}+1)(p^\rhal-p^{\rhal-1})
   +\micsm{\ 1\leq i\leq N_0}{}\,\micsm{d_{i-1}<\mu\leq d_i \ }{}
    \micsm{r_i\leq\alrh\leq\nx}{}(\tht_{\mu,\rhal}+1)(p^\rhal-p^{\rhal-1})\\
  &=&\micsm{1\leq\rhal\leq\nx}{}(\tht_{1,\rhal}+1)(p^\rhal-p^{\rhal-1})
   +\micsm{\ 1\leq i\leq N_0}{}\,\micsm{r_i\leq\alrh< r_{i+1}}{}\,
    \micsm{d_0<\mu\leq d_i \ }{}(\tht_{\mu,\rhal}+1)(p^\rhal-p^{\rhal-1})\\
  &=&\micsm{\ 1\leq i\leq N_0}{}\,\micsm{\ r_i\leq\alrh< r_{i+1}}{}\,
    \micsm{1\leq\mu\leq d_i \ }{}(\tht_{\mu,\rhal}+1)(p^\rhal-p^{\rhal-1})\\
  &=&\micsm{\ 1\leq i\leq N_0}{}\,\micsm{\ r_i\leq\alrh< r_{i+1}}{}\,
    (\ux_\rhal+1)(p^\rhal-p^{\rhal-1})\\
  &=&\micsm{1\leq\alrh\leq\nx}{}(\ux_\rhal+1)(p^\rhal-p^{\rhal-1})\\
  &=&\deg(\euD_{L|K}).                
\end{eqnarray*}
We thus conclude the proof of Key Lemma~\ref{keylemma1}
by applying Kato's criterion~Fact~\ref{KCR}.
\vskip7pt
\noindent
D) {\it Characteristic $p$ global Oort Conjecture\/}
\begin{theorem}
\label{charpOC}
{\rm (Characteristic $p$ global Oort conjecture)} \ \ 
In the notations from the Key Lemma~\ref{keylemma1},
let $Y\to X$ be a (ramified) Galois cover of complete 
smooth $k$-curves having only cyclic groups as inertia 
groups, and set $\clX_\euo:=X\times_k\euo$. Then there 
exists a $G$-cover of complete smooth $\euo$-curves 
$\clY_\euo\to \clX_\euo$ with special fiber $\,Y\to X$ 
such that the generic fiber $\clY_\kk\to \clX_\kk$ of
$\clY_\euo\to \clX_\euo$ has no essential ramification.
\end{theorem}
\begin{proof}
First, as in the case of the classical Oort Conjecture,
the local-global principle for lifting (ramified) Galois 
covers, see \nmnm{Garuti}~\cite{Ga1},~\S3, as well 
as \nmnm{Saidi}~\cite{Sa},~\S1.2, where the proofs 
of Propositions~1.2.2 and~1.2.4 are very detailed, 
reduces the proof of the Theorem~\ref{charpOC} to the 
corresponding local problem over $\euo$. Further,
exactly as in the case of the classical local Oort
Conjecture, the local problem is equivalent to the case 
where the inertia groups are cyclic $p$-groups.  One 
concludes by applying the~Key~Lemma~\ref{keylemma1}. 
\end{proof}
%
%
%
\section{Proof of Theorem~\ref{OC}}
\noindent
A) \ {\it Generalities about covers of $\,\lvP^1$\/}
\vskip2pt
\begin{notations}
\label{notanota}
We begin by introducing notations concerning 
families of covers of curves which will be used
throughout this section. Let $S$ be a separated,
integral normal scheme, e.g., $S=\Spec A$ with 
$A$ and integrally closed domain, and 
$\kbm:=\kappa(S)$ its field of rational functions.
Let $\kbm(t)\hra F$ be a finite extension of the 
rational function field $\kbm(t)$. 
\vskip2pt
1) $\lvPt S=\Proj\lvZ[t_0,t_1]\times S$ is the 
$t$-projective line over $S$, where $t=t_1/t_0$
is the canonical parameter on $\lvPt S$. In particular, 
$\lvPt S$ is the gluing of its canonical affine lines 
over~$S$, namely $\lvAt S:=\Spec\lvZ[t]\times S$ 
and $\lvA^1_{\tmu\!,\,S}:=\Spec\lvZ[\tmu]\times S$.
\vskip2pt
2) Let $\kbm(t)\hra F$ be a finite extension, and
$\clY_{t,S}\to \lvAt S$ and $\clY_{\tmu,S}\to\lvA^1_{\tmu,S}$
the corresponding normalizations in $\kbm(t)\hra F$.
Then the normalization $\clY_S\to\lvPt S$ of $\lvPt S$
in $\kbm(t)\hra F$ is nothing but the gluing of
$\clY_{t,S}\to \lvAt S$ and $\clY_{\tmu\!,\,S}\to\lvA^1_{\tmu\!,\,S}$.
\vskip2pt
3) For every $\eup\in S$ we denote by $\oli\eup\hra S$
the closure of $\eup$ in $S$ (endowed with the reduced
scheme structure). We denote by $\clO_\eup:=\clO_{S,\eup}$ 
the local ring at $\eup\in S$. We set $S_\eup:=\Spec\clO_\eup$ 
and consider the canonical morphism $S_\eup\hra S$. 
We notice that $\eup\hra S$ is both the generic fiber of 
$\oli\eup\hra S$ and the special fiber of $S_\eup\hra S$ 
at $\eup$. We get corresponding base changes:
\[
\clY_{\hhb1\oli\eup}\to\lvPt{\oli\eup}, \quad 
\clY_{S_\eup}\to\lvPt{S_\eup},\quad
  \clY_\eup\to\lvPt\eup
\]
where $\clY_\eup\to\lvPt\eup$ is both the generic fiber 
of $\clY_{\hhb1\oli\eup}\to\lvPt{\oli\eup}$ and the special
fiber of $\clY_{S_\eup}\to\lvPt{S_\eup}$.
\vskip2pt
4) Finally, affine schemes will be sometimes replaced
by the corresponding rings. Concretely, if $S=\Spec A$, 
and $\kbm={\rm Quot}(A)$, for a finite extension
$\kbm(t)\hra F$ one has/denotes:
\begin{itemize}
\vskip2pt
\item[a)] The $t$-projective line over $A$ is
$\lvPt A=\Spec A[t]\cup\Spec A[\tmu]$, 
and the normalization $\clY_A\to\lvPt A$ of $\lvPt A$ 
in $\kbm(t)\hra F$ is obtained as the gluing of 
$\Spec\clR_t\to\Spec A[t]$ and 
$\Spec\clR_{\tmu}\to\Spec A[\tmu]$, where $\clR_t$,
respectively $\clR_{\tmu}$, are the integral closures of $A[t]$, 
respectively of $A[\tmu]$, in the field extension $\kbm(t)\hra F$. 
\vskip2pt
\item[b)] For $\eup\in\Spec(A)$ one has/denotes:
$\clY_{A/\eup}\to\lvPt{A/\eup}$ and 
$\clY_{A_\eup}\to\lvPt{A_\eup}$ are the base 
changes of $\clY_A\to\lvPt A$ under $A\onto A/\eup$, 
respectively $A\hra A_\eup$; and finally, the fiber 
$\clY_{\kp\eup}\to\lvPt{\kp\eup}$ of $\clY_A\to\lvPt A$ 
is both, the special fiber of $\clY_{A_\eup}\to\lvPt{A_\eup}$ 
and the generic fiber of $\clY_{A/\eup}\to\lvPt{A/\eup}$.
\end{itemize}
\end{notations}
In the above notations, suppose that $A=\clO$ is a 
local ring with maximal ideal $\eum$ and residue 
field $\kp\eum$. Let $\clO_v$ be a valuation ring 
of $\kbm$ domination $\clO$ and having $\kp\eum=\kp v$. 
We denote by $\clY_\clO\to\lvPt\clO$ and 
$\clY_{\clO_v}\to\lvPt{\clO_v}$ the corresponding
normalizations of the the corresponding projective
lines. The canonical morphism $\Spec\clO_v\to\Spec\clO$
gives canonically a commutative diagrams dominant 
morphisms of the form:
\[
\begin{matrix}
\clY_{\clO_v}&\to&\lvPt{\clO_v}&\hhb{20}&
    \clY_\kbm&\to&\clY_{\clO_v}&\leftarrow&\clY_v\cr
\dwn{}&&\dwn{}&&\dwn{\cong}&&\dwn{}&&\dwn{}\cr
\clY_{\clO}&\to&\lvPt{\clO}&&
    \clY_\kbm&\to&\clY_{\clO}&\leftarrow&\clY_\eum\cr   
\end{matrix}
\]

We denote by $\eta_\eum\in\lvPt\eum$ the generic 
point of the special fiber of $\lvPt\clO$, and by 
$\eta_{\eum,i}\in\clY_\eum$ the generic points of the 
special fiber of $\clY_\clO$. Correspondingly, 
$\eta_v\in\lvPt v$ is the generic point of the special 
fiber of $\lvPt{\clO_v}$, and $\eta_{v,j}\in\clY_v$ are 
the generic points of the special fiber of $\clY_{\clO_v}$. 
We notice that $\lvPt v \to \lvPt\eum$ is an isomorphism
(because $\kp\eum=\kp v$), 
and by the valuation criterion for completeness, for every 
$\eta_{\eum,i}$ there exists some $\eta_{v,j}$ such that   
$\eta_{v,j}\mapsto\eta_{\eum,i}$ under $\clY_v\to\clY_\eum$.
The local ring $\clO_{\eta_v}$ of $\eta_v\in\lvPt{\clO_v}$ 
is the valuation ring of the so called \defi{Gauss valuation}
$v_t$ of $\kbm(t)$, thus $\clO_{\eta_{v,j}}$ are the valuation
rings of the prolongations $v_j$ of $v_t$ to $F$. Finally,
for every complete $k$-curve $C$ we denote by $g_C$ the
geometric genus of $C$.
\begin{lemma}
\label{preplemma}
In the above notations, let $Y_{\eum,1}\to\clY_{\eum,1}$ 
be the normalization of $\clY_{\eum,1}$. Suppose that 
$[\kappa(\eta_{\eum,1}):\kp\eum(t)]\geq[F:\kbm(t)]$ and 
that $\clY_\kbm$ is smooth. Then the following hold:
\vskip2pt
{\rm1)} The special fibers $\clY_\eum$ and $\clY_v$ are 
reduced and irreducible. 
\vskip2pt
{\rm2)} If $g_{Y_{\eum,1}}\geq g_{\clY_\kbm}$, then 
$\clY_{\clO_v}\to\lvPt{\clO_v}$ is a cover of smooth 
$\clO_v$-curves.
\end{lemma}
\begin{proof}
To 1): Let $\eta_{v,1}\mapsto\eta_{\eum,1}$.
Then $\kp{\eta_{\eum,1}}\hra\kp{\eta_{v,1}}$, thus
$[\kp{\eta_{\eum,1}}:\kp\eum(t)]\leq[\kp{\eta_{v,1}}:\kp v(t)]$.
Hence using the fundamental equality and the hypothesis 
one gets that
\[
[\kappa(\eta_{v,1}):\kp v(t)]\geq[\kappa(\eta_{\eum,1}):\kp\eum(t)]\geq
[F:\kbm(t)]={\textstyle\sum}_j[\kp{v_j}:\kp{v(t)}]e(v_j|v_t)\delta(v_j|v_t),
\]
where $e(\cdot|\cdot)$ is the ramification index and 
$\delta(\cdot|\cdot)$ is the Ostrowski defect. Hence we 
conclude that $w_1$ is the only prolongation of $v_t$ to $F$, 
and $e(v_1|v_t)=1=\delta(v_1|v_t)$. 
\vskip2pt
To 2): Since $\clY_v$ is reduced and irreducible, by
\nmnm{Roquette}~\cite{Ro1}, Satz~I, it follows that the 
Euler characteristics of the special fiber $\clY_v$ and 
that of the generic fiber $\clY_\kbm$ of $\clY_{\clO_v}$ 
are equal: 
\[
\chi(\clY_\kbm|\kbm)=\chi(\clY_v|\kp v).
\]
Since $\clY_v$ dominates $\clY_\eum$, one has
$\kappa(\clY_\eum)\hra\kappa(\clY_v)$, hence one
has $g_{\clY_\eum}\leq g_{\clY_v}$. Thus~2)
implies:
\[ 
1-g_{Y_{\eum,1}}\leq1-g_{\clY_\kbm}=\chi(\clY_\kbm|\kbm)=
     \chi(\clY_v|\kp v)\leq\chi(Y_v|\kp v)\leq\chi(Y_{\eum,1}|\kp\eum)
         =1-g_{Y_{\eum,1}}.
\]    
Hence all the above are equalities, thus finally one has
that $(\clY_v|\kp v)\leq\chi(Y_v|\kp v)$. Therefore, the
normalization $Y_v\to\clY_v$ is an isomorphism, and 
$\clY_{\clO_v}$ is smooth.
\end{proof}
In the context above, let $A=\clO$ be a {\it valuation ring\/}
and $v$ be its valuation. Let $v:=v_0\circ v_1$ be the 
valuation theoretical composition of two valuations, say 
with valuation rings $\clO_1\subset\kbm$, respectively 
$\clO_0\subset\kbm_0$, where $\kbm_0:=\kbm v_1$ is 
the residue field of~$v_1$. Then $\kbm v=:k:=\kbm_0v_0$
is the residue field of both $v$ and $v_0$. Let $t \in F$ 
is a fixed function, and $t_0:=tv_{0,t}$ be the residue 
of $t$ with to the Gauss valuation $v_{0,t}$ on $\kbm(t)$. 
Suppose that the following hold:
\vskip2pt
$\hhb2$i) The special fiber $\clY_{1,\hhb1s}$ of the normalization 
$\clY_1\to\lvPt{\clO_1}$ of $\lvPt{\clO_1}$ in $\kbm(t)\hra F$ 
is irreducible. Thus $v_{1,t}$ has a unique prolongation 
$w_1$ to $F$, and $F_0:=Fw_1=\kappa(\clY_{1,\hhb1s})$.
\vskip2pt
ii) The special fiber $\clY_{0,\hhb1s}$ of the normalization 
$\clY_0\to\lvPt{\clO_0}$ of $\lvPt{\clO_0}$ in 
$\kbm_0(t_0)\hra F_0$ is irreducible. Thus $v_{0,t_0}$ 
has a unique prolongation $w_0$ to $F_0$, and 
$F_0w_0=\kappa(\clY_{0,\hhb1s})$.
\begin{lemma} 
\label{transsmooth}
{\rm(Transitivity of smooth covers)} 
In the above notations, suppose that the hypotheses~i),~ii)
are satisfied. Set $w:=w_0\circ w_1$, and let 
$\,\clY\to\lvPt{\clO}$ be the normalization of $\lvPt{\clO}$ 
in $\kbm(t)\hra F$. Then $w$ is the unique prolongation 
of $v_t$ to $F$, and the following hold:
\begin{itemize}
\item[{\rm 1)}] The base change of $\,\clY$ under 
$\clO\hra\clO_{v_1}$
is $\,\clY_1=\clY\times_\clO\clO_{v_1}$ canonically, thus
$\clY_{1,\hhb1s}=\clY_{\eum_1}$ is the fiber of $\,\clY$ at
the valuation ideal $\eum_1\in\Spec\clO$ of $v_1$. 
\vskip2pt
\item[{\rm 2)}] Let $\clY_{\clO_0}\to\lvPt{\clO_0}$
be the base change of $\,\clY\to\lvPt\clO$ under the
$\clO\onto\clO_0$. Then~$\,\clY_{\eum_1}$ is the generic fiber 
of $\clY_{\clO_0}$ and $\clY_0\to\lvPt{\clO_0}$ is the 
normalization of $\clY_{\clO_0}\to\lvPt{\clO_0}$.
\end{itemize}
In particular, $\clY$ is a smooth $\clO$-curve if and only 
if $\clY_1$ is a smooth $\clO_1$ curve and  $\clY_0$ is a 
smooth $\clO_0$-curve.
\end{lemma}
\begin{proof} Klar, by the discussion above, and \nmnm{Roquette}
\cite{Ro1},~Satz~I, combined with the fact that a projective 
curve is smooth if and only if its arithmetic genus equal its 
geometric genus. 
\end{proof}

\vskip5pt
\noindent
B) \ {\it A specialization result\/}
\vskip5pt
We begin by recalling the following two well known facts. 
The first one is by \nmnm{Katz} (and \nmnm{Gabber})~\cite{Ka}: 
Let $k$ be an algebraically closed field with $\chr(k)=p$. Then 
the localization at $t=0$ defines a bijection between the finite 
Galois $p$-power degree covers of $\lvPt k$ unramified 
outside $t=0$ and the finite Galois $p$-power extensions 
$k[[t]]\hra k[[z]]$, and this bijection preserves the ramification 
data. Thus given a cyclic $\lvZ/p^\mx\,$-cover $k[[t]]\hra k[[z]]$, 
there exists a unique cyclic $\lvZ/p^\mx\,$ cover of complete 
smooth curves $Y_k\to\lvPt k$ which is branched only at 
$t=0$ (thus totally branched there) such that $k[[t]]\hra k[[z]]$ 
is the extension of local rings of $Y\to\lvPt k$ above $t=0$. 
We will say that $Y\to\lvPt k$ is the \defi{KG\hhb1-\hhb1cover} for 
$k[[t]]\hra k[[z]]$. 
The second fact is the local-global principle for the Oort
Conjecture, see e.g.\ \nmnm{Garuti}~\cite{Ga1},~\S3,
\nmnm{Saidi}~\cite{Sa},~\S1.2, especially~Proposition~1.2.4, 
which among other things imply:
\begin{LGP} 
\label{LGP}
Let $k[[t]]\hra k[[z]]$ be a $\lvZ/p^\mx$-extension and
$Y_k\to\lvPt k$ be its KG-cover. Further let $W(k)\hra R$ be 
a finite extension of $W(k)$. Then the $\lvZ/p^\mx$-extension
$k[[t]]\hra k[[z]]$ has a smooth lifting over $R$ if and only 
if the $\lvZ/p^\mx$-cover $Y_k\to\lvPt k$ has a smooth 
lifting over $R$.
\end{LGP}
Next let $\mx$ be a fixed positive integer, and 
consider a finite sequences of positive numbers 
$\uix:=(\ix_1\leq\dots\leq\ix_\mx)$ satisfying: $1\leq\ix_1$ 
is prime to $p$, and $\ix_{\alp+1}=p\hhb1\ix_\alp+\epsilon$
with $\epsilon\geq0$ and $\epsilon$ prime to $p$ if $\epsilon>0$.
For such a sequence $\uix$, let $|\uix|=\ux_1+\dots+\ux_\mx$ 
and consider ${\undr P}_{\uix}=\big(P_1,\dots,P_\mx)$ a 
sequence of generic polynomials $P_\nu=P_\nu(\tmu)$ of 
degrees $\deg\big(P_\alp\big)=\ix_\alp$ for $1\leq\alp\leq\mx$.
In other words, all the coefficients $a_{\nu,\rhal}$, 
$1\leq\nu\leq\mx$, $1\leq\rhal\leq\ix_\nu$ of the polynomials 
$P_\alp$ are independent free variables over $\kzr:=\oli\lvF_p$. 
Let $A_{\uix}:=\kzr[(a_{\nu,\rhal})_{\nu,\rhal}]$ be the 
corresponding polynomial ring and $\lvAix=\Spec A_{\uix}$ 
the resulting affine space over $\kzr$. 
\vskip2pt
For every $x\in\lvAix$ let $k_x$ be any algebraically 
closed field extension of $\kzr$, and $\oli x\in\lvAix(k_x)$ 
be a $k_x$-rational point of $\lvAix$ defined by a 
$\kzr$-embedding $\phi_x:\kp x\hra k_x$. Let 
${\undr p}_{\uix,x}=(p_{1,x},\dots,p_{\mx,x})$ and 
${\undr p}_{\uix,\oli x}=(p_{1,\olix},\dots,p_{\mx,\olix})$ be 
the images of ${\undr P}_\uix$ over $\kp x$, respectively
$k_x$. Then one has virtually by definitions that
$p_{\nu,\olix}=\phi_x(p_{\nu,x})$, thus 
${\undr p}_{\uix,\olix}=\phi_x({\undr p}_{\uix,x})$.
In particular, if $\deg(P_\alp)=\deg(p_\alp)$ for all $\alp$, 
then ${\undr p}_{\uix,x}$ gives rise to a cyclic extension 
$k_x[[t]]\hra k_x[[z_x]]$ of degree $p^\mx$ and upper 
jumps $\uix=(\ux_1,\dots,\ux_\mx)$, and canonically to
its KG-cover $Y_{k_x}\to\lvPt{k_x}$.
\begin{definition} For $\kzr\hra k_x$ as above, let 
$k_x[[t]]\hra k_x[[z_x]]$ be a cyclic $\lvZ/p^\mx$-extension 
and $Y_{k_x}\to\lvPt{k_x}$ be its $KG$-cover. We say that
$k_x[[t]]\to k_x[[z_x]]$ is an $\uix$-\defi{extension} at $x\in\lvAix$ 
and that $Y_{k_x}\to\lvPt{k_x}$ is an $\uix$-KG-\defi{cover} 
at $x\in\lvAix$, if $k_x[[t]]\hra k_x[[z_x]]$ has 
$\uix=(\ix_1,\dots,\ix_\mx)$ as upper ramification jumps.    
\end{definition}
\begin{notations}
\label{notalast}
We denote by $\Sigma_\uix\subset\lvAix$ the set of all 
$x\in\lvAix$ which satisfy: There exists some mixed 
characteristic valuation ring $R_x$ with residue field 
$k_x$ such that some $\uix$-KG-cover $Y_{k_x}\to\lvPt{k_x}$ 
has a smooth lifting over $R_x$.
\end{notations}
%
%
\begin{proposition}
\label{specprop}
In Notations~\ref{notalast}, suppose that $\Sigma_\uix\subseteq\lvAix$ 
is Zariski dense. Then there exists an algebraic integer 
$\pi_{\uix}$ such that for every algebraically closed field 
$k$ of characteristic  $\chr(k)=p$ one has: Every $\uix$-KG-cover 
$Y_k\to\lvPt k$ has a smooth lifting over $W(k)[\pi_{\uix}]$.
\end{proposition}
\begin{proof} The proof is quite involved, and has 
two main steps as follows:
\vskip7pt
\noindent
{\bf Step 1}. {\it Proving that the generic point 
                     $\eta_\uix\in\lvAix$ lies in $\Sigma_\uix$\/}
\vskip5pt
Let $\eu U$ be an ultrafilter on $\Sigma$ 
which contains all the Zariski open subsets of $\Sigma$. 
(Since $\Sigma$ is Zariski dense in the irreducible 
scheme $\lvAix\!$, any Zariski open subset of $\Sigma$ 
is dense as well, thus ultrafilter $\eu U$ exist.) 
Let $k_x\to\Theta_x\subset R_x$ be any set of 
representatives for  for $R_x$. Consider the following 
ultraproducts index by $\Sigma$: 
\[
\str2k:=\ultrb{k_x}\to\str1\Theta:=\ultrb{\Theta_x}
\subset \Wstr:=\ultrb{W(k_x)}\hra\ultrb{R_x}=:\str2R.\ \
\]
By general model theoretical principles, it follows that 
$\str2R$ is a valuation ring having residue field equal to 
$\str2k$, and $\str1\Theta\subset\str2R$ is a system of 
representatives for the residue field $\str2k$ of $\str2R$.
\vskip5pt
Next, coming to geometry, by general model theoretical
principles, it follows that the family of $\lvZ/p^\nx$-covers 
$Y_x\to\lvPt{k_x}$ with upper ramification jumps 
$\uix=(\ix_1,\dots,\ix_f)$ gives rise to a ${\eu U}$-\defi{generic} 
$\lvZ/p^\mx$-cover $Y_{\str2k}\to\lvPt{\str2k}$ of complete 
smooth $\str2k$-curves with upper ramification jumps~$\uix$. 
Precisely, setting $\str1{p_\nu}:=(p_{\nu,x})_x/{\eu U}$,
the system of polynomials 
$\str1{\undr p}_{\uix}=
         \big(\str1{p_1}(\tmu),\dots,\str1{p_\mx}(\tmu)\big)$
defines the local extension $\str2k[[t]]\hra\str2k[[z]]$ of 
$Y_{\str2k}\to\lvPt{\str2k}$ at $t=0$. Moreover, consider 
the $\str2k$-rational point $\str3\phi_{\eta_\uix}$ of $\lvAix$ 
defined by 
\[
\str3\phi_{\eta_\uix}:\kp{\eta_\uix}\to\str2k, \quad
\str3\phi_{\eta_\uix}(a_{\nu,\rhal}):=
                  \big(\phi_x(a_{\nu,\rhal})\big)_{\!x\,}/{\eu U}.
\leqno{\indent(*)}                  
\]
Then $\str1{\undr p}_\uix=\str3\phi_{\eta_\uix}({\undr P}_\uix)$,
which means that $\str3\phi_{\eta_\uix}$ is the 
$\str2k$-rational point of $\lvAix$ defining $\str1{\undr p}_\uix$.
\vskip4pt
Again, by general model theoretical principles for 
ultraproducs of (covers of) curves, the family of the 
$\lvZ/p^\mx$-covers $\clY_{R_x}\to\lvPt{R_x}$ with special 
fiber $Y_x\to\lvPt{k_x}$ gives rise to a $\lvZ/p^\mx$-cover
$\clY_{\str2R}\to\lvPt{\str2R}$ of complete smooth 
$\str2R$-curves, with $Y_{\str2k}\to\lvPt{\str1k}$ as special 
fiber. 
\vskip4pt
Let $\lvAix\hra\lvPix$ be the canonical
embedding of the affine $\kzr$-space 
$\lvAix:=\Spec \kzr[(a_{\nu,\rhal})_{\nu,\rhal}]$
into the corresponding projective $\kzr$-space
$\lvPix:=\Proj \kzr[t_0, (t_{\nu,\rhal})_{\nu,\rhal}]$
via the $t_0$-dehom\-ogen\-iza\-tion 
$a_{\nu,\rhal}=t_{\nu,\rhal}/t_0$. Letting 
$\Srz:=\lvZ^{^{\rm nr}}_p$ be the maximal unramified 
extension, and
\[
\lvAixZ=\Spec\Srz[(a_{\nu,\rhal})_{\nu,\rhal}]\ \ {\rm and}
\ \ \lvPixZ=\Proj\Srz[t_0, (t_{\nu,\rhal})_{\nu,\rhal}],
\] 
the embedding $\lvAix\hra\lvPix$ is the special 
fiber of $\lvAixZ\hra\lvPixZ$. Notice that 
$\str3\phi_{\eta_\uix}:A_\uix\to\str2k$ gives rise via 
$\str2k\to\str1\Theta$ canonically to an embedding 
of $\Srz$-algebras defined by
\[
\str3\phi_{\Scz}:A_{\uix,\Scz}:=
     \Srz[(a_{\nu,\rhal})_{\nu,\rhal}]\hra\str2R,\quad 
        a_{\nu,\rhal}\mapsto\str3{\phi_{\eta_\uix}}(a_{\nu,\rhal}).
\]

Let $\V$ be a projective normal $\Srz$-scheme with function
field $\kappa(\V)$ embeddable in ${\rm Quot}(\str2R)$, say 
via $\kappa(\V)\hra{\rm Quot}(\str2R)$, such that the following
are satisfied: 
\vskip4pt
1) Let $\eupz\in \V$ be the center of $\wstr$ 
of $\str2R$ on $\V$ induced by $\kappa(\V)\hra\str2R$, and 
$\clOz$ the local ring of $\eupz$. Then the $\lvZ/p^\mx$-cover 
of complete smooth $\str2R$-curves $\clY_{\str2R}\to\lvPt{\str2R}$ 
is defined over $\clOz$. 
\vskip4pt
2) The image of $\str3\phi_\Scz:A_{\uix,\Scz}\to\str2R$ is 
contained in the image of $\kappa(\V)\hra\str2R$, and the 
resulting embedding $A_{\uix,\Scz}\hra\kappa(\V)$ is defined 
by some proper morphism 
\[
\V\to\lvPixZ.
\]

We notice that condition~1) means that there exists a
$\lvZ/p^\mx$-cover of complete smooth $\clOz$-curves 
$\clY_{\clOz}\to\lvPt{\clOz}$ such that 
$\clY_{\str2R}\to\lvPt{\str2R}$ is the base change 
of $\clY_{\clOz}\to\lvPt{\clOz}$ under $\clOz\hra\str2R$.
In particular, if $\clOz\to\kp{\eupz}$ is the 
residue field of $\clOz$, then the special fiber 
$Y_{\eupz}\to\lvPt{\eupz}$ of 
$\clY_{\clOz}\to\lvPt{\clOz}$ is a $\lvZ/p^\mx$-cover 
of complete smooth $\kp{\eupz}$-curves whose 
base change under $\kp{\eupz}\hra\str2k$ is canonically 
isomorphic to $Y_{\str2k}\to\lvPt{\str2k}$. In other words, 
the embedding $\str3\phi_{\eta_\uix}: A_\uix\hra\str2k$ 
defined by~$(*)$ above factors
through $A_\uix\hra\kp{\eta_\uix}\hra\kp{\eupz}$,
and $\eupz\in \V$ is mapped to the generic point 
$\eup\mapsto\eta_\uix$ of the special fiber 
$\eta_\uix\in\lvPix\hra\lvPixZ$ under $\V\to\lvPixZ$.
\vskip2pt
Recall that $\oli\eupz\subset \V$ the Zariski 
closure of $\eupz$ in $\V$ viewed as a closed 
$\Srz$-subscheme of~$\V$ endowed with the 
reduced scheme structure.
Since $\eup\mapsto\eta_\uix$, one has that 
$\kp{\eta_\uix}\hra\kp{\eupz}$, hence $\kp{\eupz}$ 
has characteristic~$p$, and $\eupz$ lies in the special 
fiber $\V_{\kzr}$ of $\V\!$. We conclude that 
$\oli\eupz\subset \V_{\kzr}$. 
\vskip2pt
Next, if $\eupz$ has codimension $>1$, let 
$\tlV \to \V$ be the normalization of the blowup 
of $\V$ along the closed $\Srz$-subscheme $\oli\eupz$. 
Let $E_1,\dots, E_r\subset\tlV$ be the finitely 
many irreducible components of the preimage of the 
exceptional divisor of the blowup. Then the generic 
points $\eur_i$ of the $E_i$, $i=1,\dots,r$ are 
precisely the points of codimension one of $\tlV$ 
which map to $\eupz$ under $\tlV\to \V$, and 
$\cup_i E_i$ is the preimage of $\oli\eupz$ in $\V$. 
Further, if $\eur=\eur_i$ is fixed, and $\tlclO$ 
is the local ring of $\eur\in \tlV$ and $\kp{\eur}$ 
is its residue field, it follows that $\clOz\hra\tlclO$
and $\kp{\eupz}\hra\kp{\eur}$ canonically. Recall that by
the property~1) above, $\clY_{\clOz}\to\lvPt{\clOz}$ is a 
$\lvZ/p^\mx$-cover of smooth $\clOz$-curves with special 
fiber $Y_{\eupz}\to\lvPt{\eupz}$, whose base change
under $\kp{\eupz}\hra\str2k$ is $Y_{\str2k}\to\lvPt{\str2k}$.
Therefore, the base change $\clY_{\tlclO}\to\lvPt{\tlclO}$ 
of $\clY_{\clOz}\to\lvPt{\clOz}$ defined by the inclusion 
$\clOz\hra\tlclO$ is a $\lvZ/p^\mx$ cover of proper 
smooth $\tlclO$-curves whose special fiber 
$Y_{\eur}\to\lvPt{\eur}$ is the base change 
of $Y_{\eupz}\to\lvPt{\eupz}$ under 
$\kp{\eupz}\hra\kp{\eur}$. Hence choosing any
$\kp{\eupz}$-embedding $\kp{\eur}\hra\str2k$, we get 
that the special fiber $Y_{\eur}\to\lvPt{\eur}$
becomes $Y_{\str2k}\to\lvPt{\str2k}$ under 
$\kp{\eur}\hra\str2k$.
\vskip2pt
Hence by replacing $\V$ by $\tlV$ if necessary, we
can suppose that $\eupz\in \V$ has codimension one,
or equivalently, that $\oli\eupz\subset \V_{\kzr}$ is an
irreducible component of $\V_{\kzr}$.
\vskip4pt
By de Jong's theory of alterations \nmnm{de~Jong}
\cite{dJ},~Theorem~6.5, there exists a finite extension of 
discrete valuation rings $\Srz\hra\Srp:=\Srz[\piz]$ with 
$\piz$ any uniformizing parameter of $\Srp$ and a dominant 
generically finite proper morphism $\W\to \V$ of projective 
$\Srz$-schemes with $\W$ {\it strictly semi-stable\/} over
$\Srp$, i.e., the generic fiber of $\W$ is a smooth projective
variety over ${\rm Quot}(\Srp)$, the special fiber $\W_{\!\kzr}$
is reduced and satisfies: If $\W_{\!\kzr,j}$, $j\in J$ is any 
set of $|J|$ distinct irreducible components of $\W_{\!\kzr}$, 
then $\cap_j \W_{\!\kzr,j}$ is a smooth subscheme of $\W$ 
of codimension $|J|$. Hence the sequence of dominant 
proper morphisms of projective $\Srz$-schemes
\[
\W\to \V\to\lvPixZ
\]
satisfies: Let $\eupq\in \W$ denote 
a fixed preimage of $\eupz$. Then $\eupq$ has codimension
one, because $\eupz$ does so. Further, the local ring 
$\clOq$ of $\eupq\in \W$ as a point of $\W$ dominates 
the local ring $\clOz$ of $\eupz\in \V$, thus one has a 
canonical inclusion $\clOz\hra\clOq$ which gives rise
to a canonical inclusion of the residue fields 
$
\kp{\eta_\uix}\hra\kp{\eupz}\hra\kp{\eupq} \ \
\hbox{corresponding to}\ \ \euq\to\eup\to\eta_\uix.
$ 
Recall that $\clY_{\clOz}\to\lvPt{\clOz}$ is a $\lvZ/p^\mx$-cover 
of smooth $\clOz$-curves with special fiber 
$Y_{\eupz}\to\lvPt{\eupz}$ whose base change under 
$\kp{\eupq}\hra\str2k$ is $Y_{\str2k}\to\lvPt{\str2k}$.
Let $\clY_{\clOq}\to\lvPt{\clOq}$ be the 
base change of $\clY_{\clOz}\to\lvPt{\clOz}$ under
$\clOz\hra\clOq$. Then $\clY_{\clOq}\to\lvPt{\clOq}$ is a
$\lvZ/p^\mx$-cover of proper smooth $\clOq$-curves 
whose special fiber $Y_{\eupq}\to\lvPt{\eupq}$ 
is the base change of $Y_{\eupz}\to\lvPt{\eupz}$ 
under $\kp{\eupz}\hra\kp{\eupq}$. Again, choosing any
$\kp{\eupz}$-embedding of $\kp{\eupq}\hra\str2k$, one
gets that the base change of the special fiber 
$Y_{\eupq}\to\lvPt{\eupq}$ under $\kp{\eupq}\hra\str2k$
becomes $Y_{\str2k}\to\lvPt{\str2k}$. This means that the 
embedding $\str3\phi_{\eta_\uix}: A_\uix\hra\str2k$ 
defined at~$(*)$ above factors through 
$A_\uix\hra\kp{\eta_\uix}\hra\kp{\eupz}\hra\kp{\eupq}$,
reflecting the fact that $\eupq\mapsto\eupz\mapsto\eta_\uix$.
In other words, there exists a $\str2k$-rational point
$\str3\phi_\euq:\kp\euq\to\str2k$ such that the given 
$\str2k$-rational point $\str3\phi_{\eta_\uix}:\kp{\eta_\uix}\to\str2k$ 
defined by $\str3\phi_{\eta_\uix}: A_\uix\to\Spec\str2k$ is of the form
\[
\str3\phi_{\eta_\uix}=
    \str3\phi_\euq\circ(\kp{\eta_\uix}\hra\kp\euq).
\leqno{\indent(**)}
\]
%
\vskip4pt
\noindent
{\bf Step 2}. {\it Finishing the proof of Proposition~\ref{specprop}\/}
\vskip5pt
Let $\llll:=\kappa(W)$ denote the function field of $W$, and
$F:=\kappa(\clY_{\clOq})$ be the function field of $\clY_{\clOq}$.
Then $\clY_{\clOq}\to\lvPt{\clOq}$ has as generic fiber a 
$\lvZ/p^\mx$-cover of complete smooth $\llll$-curves
$\clY_\llll\to\lvPt{\llll}$, and gives rise to a $\lvZ/p^\mx$ 
extension of function field in one variable $\llll(t)\hra F$. 
Since $\clOq$ is a (discrete) valuation ring, and 
$\clY_{\clOq}\to\lvPt{\clOq}$ is a cover of smooth 
$\clOq$-curves, it follows by the discussion in subsection~A), 
that $\clY_{\clOq}\to\lvPt{\clOq}$ is precisely the normalization 
of $\lvPt{\clOq}$ in the function field extension $\llll(t)\hra F$. 
Notice that $\lvPt{\llll}$ is the generic fiber of $\lvPt W$, and 
consider 
\[
\clY_W\to\lvPt W
\] 
the normalization of $\lvPt W$ in the field extension
$\llll(t)\hra F$. We notice that the base change of 
$\clY_W\to\lvPt W$ under $\Spec\clOq\hra W$ is 
precisely $\clY_{\clOq}\to\lvPt{\clOq}$.
\begin{lemma}
\label{speclemma}
Let $x\in\lvAix$ be such that the image 
${\undr p}_{\uix,x}=(p_{1,x},\dots,p_{\mx,x})$ of
${\undr P}_\uix=(P_1,\dots,P_\mx)$ under 
$A_\uix\to\kp x$ satisfies $\deg(p_{\nu,x})=\deg(P_\nu)$ 
for all $\nu=1,\dots,\mx$. Let $y\in\oli\euq$ be a 
preimage of $x\in\lvAix\subset\lvPix$ under 
$\oli\euq\to\oli\eup\to\lvPix$ and $\clO_v$ be a 
valuation ring dominating $\clOy$ with $\kp v=\kp\y$. 
Then $\clY_{\clO_v}\to\lvPt{\clO_v}$ is a cover of 
smooth curves.
\end{lemma}
\begin{proof} Recall that $\clY_{\clOy}\to\lvPt{\clOy}$ 
is the base change of $\clY_W\to\lvPt W$ under the
canonical embedding $\Spec\clOy\hra W$, and in 
particular, $\clY_{\clOy}\to\lvPt{\clOy}$ is the normalization
of $\lvPt{\clOy}$ in the field extension $\llll(t)\hra F$.
Since $y\in\euq$, and the geometric fiber 
$\clY_\euq\to\lvPt\euq$ of $\clY_{\oli\euq}\to\lvPt{\oli\euq}$
is a $\lvZ/p^\mx$-cover of smooth complete curves,
the same holds correspondingly, if one replaces
$\clOy=\clO_{W,y}$ by $\euoy:=\clO_{\oli\euq,y}=\clOy/\euq$,
and $\llll(t)\hra F$ by $\kp\euq(t)\hra F_\euq$, where
$F_\euq:=\kappa(\clY_\euq)$ is viewed as function field 
over $\kp\euq$. Recall 
that the local extension $\kp\euq[[t]]\hra\kp\euq[[z_\euq]]$ 
of $\clY_\euq\to\lvPt\euq$ at $t=0$ is defined by the 
image ${\undr p}_{\uix,\euq}$ of ${\undr P}_\uix$ 
under the canonical embedding 
$A_\uix\hra\kp{\eta_\uix}\hra\kp\euq$. On the other hand, if
$\euox$ denotes the local ring of $x\in\lvAix\subset\lvPix$
then $A_\uix\subset\euox$ and $\euoy$ dominates~$\euox$.
Hence $A_\uix\hra\euoy$ and therefore, 
${\undr p}_{\uix,\euq}$ is defined over $\euoy$. 
Further, by the commutativity of the diagrams
\[
\begin{matrix}
A_\uix&\hra&\euoy&&{\undr P}_\uix&\mapsto&{\undr p}_{\uix,\euq}\cr 
\dwn{}&    &\dwn{}&\hhb{20} &\dwn{}           &             &\dwn{}\cr
\kp x  &\hra&\kp\y & &{\undr p}_{\uix,x}&\mapsto&{\undr p}_{\uix,\y}\cr
\end{matrix}
\]
it follows that the image of ${\undr p}_{\uix,\euq}$ under
the residue homomorphism $\euoy\to\kp\y$ equals the
image of ${\undr p}_{\uix,x}$ under $\kp x\hra\kp\y$.
Thus by the functoriality of the Artin--Schreier--Witt theory, 
it follows that every irreducible component of the special 
fiber of $\lvPt{\euoy}$ dominates the KG-cover of 
$\lvPt{\kp\y}$ defined by ${\undr p}_{\uix,y}$. Since 
$\deg(p_{\nu,y})=\deg(p_{\nu,x})=\deg P_\nu$ for all $\nu$, 
the latter cover must have degree $p^\mx$ and upper ramification 
jumps $\uix=(\ix_1,\dots,\ix_\mx)$. In particular, we can
apply~Lemma~\ref{preplemma}, and conclude that
the special fibers $\clY_y$ and $\clY_v$ are reduced 
and irreducible. 
\vskip2pt
In order to conclude, we notice that by the discussion
above, the normalization $Y_\y\to\clY_\y$ dominates the 
$\uix$-KG-cover of $\lvPt\y$ defined by ${\undr p}_{\uix,y}$. 
Since every $\uix$-KG-cover has as genus a constant  
depending on $\uix$ only, thus including the generic fiber 
it follows that $g_{Y_\y}\geq g_{\clY_\euq}$. We thus 
conclude the proof of Lemma~\ref{speclemma} 
by applying~Lemma~\ref{preplemma}.
\end{proof}

Coming back to the proof of Proposition~\ref{specprop}
we proceed as follows. Let $k$ be any algebraically 
closed field with $\chr(k)=p$, and $Y_k\to\lvPt{k}$ be 
an $\uix$-KG-cover, say with local ring extension 
$k[[t]]\hra k[[z]]$ at $t=0$ defined by 
${\undr p}_\uix=\big(p_1,\dots,p_\mx)$.
\vskip2pt
In notations as introduced right before 
Lemma~\ref{speclemma}, let $x\in\lvAix$ and $\phi_x:\kp x\to k$ 
be such that $\phi_x({\undr p}_{\uix,x})={\undr p}_\uix$.  Since 
$\oli\euq\to\oli\eup\to\lvPix$ is dominant and proper, 
there exists a preimage $\y\in\oli\eupq$ of $x$ such 
that $\kp x\hra\kp\y$ is finite. Since $k$ is algebraically 
closed, there is a $\kp x$-embedding $\phi_\y:\kp\y\hra k$ 
such that $\phi_x=\phi_y\circ(\kp x\hra\kp y)$. In particular, 
if ${\undr p}_{\uix,\y}$ is the image of ${\undr p}_{\uix,x}$ 
under $\kp x\hra\kp\y$, then ${\undr p}_\uix=
\phi_\y({\undr p}_{\uix,y})$.
\vskip4pt
Let $\W_{\kzr,j}$, $j\in J$, be the irreducible components 
of $\W_{\kzr}$ which contain~$y$, and 
$W_J:=\cap_j W_{\kzr,j}$. Then $\W_J$  is a smooth 
$\kzr$-subvariety $\W_J\subset \W_{\!\kzr}$, and the
following hold, see e.g., \nmnm{de~Jong}~\cite{dJ}, 
section~2.16 and explanations thereafter: Let $\clOy$ 
be the local ring of $y\in \W$. There exists a system of 
regular parameters $(u_1,\dots,u_N)$ of $\clOy$ which 
satisfy:
\vskip2pt
$\hhb2$i) $u_j$ defines locally at $\y$ the equation of 
$\W_{\kzr,j}$ and $\piz=u_1\dots u_{n_\y}$.
\vskip2pt
ii) $(u_j)_{n< j \leq N}$ give rise to a regular system of 
parameters at $\y\in \W_J$ in $\clOy/(u_1,\dots,u_{n_\y})$.
\vskip4pt
Let $\eur:=(u_1-u_i)_i\subset\clOy$ be the ideal generated
by all the $u_1-u_i$, $1\leq i\leq N$. Then $\eur$ is a 
regular point with $(u_1-u_i)_i$ a regular system of
parameters, and 
\[
\Srp_\y:=\clOy/\eur
\]
is a discrete valuation ring having 
$\pi_\y:=u_1\,({\rm mod}\,\eur)$ as uniformizing 
parameter, and $\pi_\y^{n_\y}=\piz$. In particular, if 
$\piz$ was an algebraic integer, then so is $\pi_\y$.
\vskip2pt
Let $v_1$ be a valuation of $F$ with center $\eur$, 
and residue field equal to $\kp\eur={\rm Quot}(\Srp_\y)$.  
And further, let $v_0$ be the canonical valuation of 
$\clO_0:=\Srp_y$. Then the valuation ring $\clO_v$ of the 
valuation $v:=v_0\circ v_1$ dominates $\clOy$ and has 
$\kp v=\kp\y$. Hence by Lemma~\ref{speclemma} 
above, it follows that $\clY_{\clO_v}\to\lvPt{\clO_v}$ 
is a $\lvZ/p^\mx$-cover of smooth $\clO_v$-curves. 
Hence by Lemma~\ref{transsmooth}, it follows that
$\clY_{\Srp_y}\to\lvPt{\Srp_y}$ is a $\lvZ/p^\mx$-cover 
of smooth $\Srp_y$-curves.
\vskip2pt
Let $n_\uix={\rm l.c.m.}(n_\y)_\y$, and notice that  
$n_\uix$ is bounded by $n\hhb1!$, where 
$n=\dim(W)-1$. Choose a fixed algebraic integer
$\piz$ such that $\Srp=\Srp_0[\piz]$, and let $\pi_\uix$ 
be defined by $\pi_\uix^{n_\uix}=\piz$. Then there are
canonical embeddings
$\Srp_\y\hra W(\oli\kappa_\y)[\pi_\y]\hra W(k)[\pi_\uix]$,
and the base change of $\clY_{\Srp_\y}\to\lvPt{\Srp_\y}$
under $\Srp_\y\hra W(k)[\pi_\uix]$ is a
$\lvZ/p^\mx$-cover of smooth $W(k)[\pi_\uix]$-curves
\[
\clY_{W(k)[\pi_\uix]}\to\lvPt{W(k)[\pi_\uix]}
\]
with special fiber the $\uix$-KG-cover $Y_k\to\lvPt k$ 
of the given cyclic $\lvZ/p^\mx$-extension $k[[t]]\hra k[[z]]$. 
\vskip2pt
This concludes the proof of Proposition~\ref{specprop}.
\end{proof}

%
\vskip5pt
\noindent
C) \ {\it The strategy of proof for Theorem~\ref{OC}\/}
\vskip2pt
We begin by recalling that there are several forms of the 
Oort Conjecture (OC) which are all equivalent, 
see e.g.\ \nmnm{Saidi}~\cite{Sa},~\S3.1, for detailed proofs.
\vskip2pt
Let $k$ be an algebraically closed field with 
${\rm char}(k)=p>0$. Let $W(k)$ be the ring of Witt 
vectors over $k$, and $W(k)\hra R$ denote finite 
extension of discrete valuation rings. We consider 
the following two situations, which are related to two 
variants of OC:
\vskip5pt
a) $Y\to X$ is a finite (ramified) $G$-cover of complete smooth 
$k$-curves such that the inertia groups at all closed points 
$y\in Y$ are cyclic. 
\vskip5pt
b) $\clX_R$ is a complete smooth $R$-curve with special fiber 
$X$, and $Y\to X$ is as a (ramified) $G$-cover of complete 
smooth curves as in~case~a) above.
\vskip5pt
We say that {\it OC holds over~$R$\/} in case a) or b), if 
there exists a $G$-cover of complete smooth $R$-curves 
$\clY_R\to\clX_R$, with $\clX_R$ the given one in~case~b), 
having the $G$-cover $Y\to X$ as special fiber. And given
a cyclic extension $k[[t]]\hra k[[z]]$, we say that the {\it local 
OC holds over $R$\/} for $k[[t]]\hra k[[z]]$, if there exists a 
smoth smooth lifting $R[[T]]\hra R[[Z]]$ of $k[[t]]\hra k[[z]]$.
%
%
\begin{fact}
\label{equivOC}
(\nmnm{Saidi}~\cite{Sa}, \S3.1) The following hold:
\vskip2pt
{\rm 1)} {\it Local global principle for OC.\/} Let $Y\to X$ be a finite 
$G$-cover with cyclic inertia groups, and for $y\mapsto x$,
let $k[[t_x]]\hra k[[t_y]]$ be the corresponding extension of
local rings. Let $\clX_R$ be some complete smooth $R$-curve 
with special fiber $X$. Then the following are equivalent:
\vskip2pt \ \ \ 
$\hhb2${\rm i)} There is a $G$-cover of complete smooth 
$R$-curves $\clY_R\to\clX_R$ with special fiber $Y\to X$.
\vskip2pt \ \ \ 
{\rm ii)} For all $y\mapsto x$, the local cyclic extension
$k[[t_x]]\hra k[[t_y]]$ has a smooth lifting over $R$.
\vskip4pt
{\rm 2)} {\it Equivalent forms of OC.\/} The following assertions 
are equivalent:
\begin{itemize}
\item[{\rm a)}] OC holds for all $G$-covers $Y\to X$ as in case a), or b).
\vskip2pt
\item[{\rm b)}] OC holds for $G$ cyclic and $X=\lvPt k$.
\vskip2pt
\item[{\rm c)}] OC holds for $G$ cyclic and $X=\lvPt k$
and $Y\to\lvPt k$ branched at $t=0$ only.
\vskip2pt
\item[{\rm d)}] The local OC holds.
\vskip2pt
\item[{\rm e)}] Any of the assertions above, but restricted to cyclic
$p$-groups as inertia groups.
\end{itemize}
\end{fact}
\vskip5pt
Thus in order to prove Theorem~\ref{OC} from Introduction,
we can proceed as follows: Let $Y\to X$ be a given 
$G$-cover of projective smooth $k$-curves, with branch
locus $\Sigma\subset X$. Then for a given algebraic 
integer $\pi$ and $R:=W(k)[\pi]$, and a smooth model
$\clX_R$ of $X$ over $R$ one has: The OC holds for 
$Y\to X$ over $R$ iff the local OC holds for the local
cyclic extension $k[[t_x]]\hra k[[t_y]]$ over $R$ for all
$x\in\Sigma$. Further, the local OC holds for a fixed local 
cyclic extension $k[[t_x]]\hra k[[t_y]]$ over $R$ if and only 
if the local OC holds over $R$ for the $p$-power sub-extension
$k[[t_x]]\hra k[[z_x]]$ of $k[[t_x]]\hra k[[t_y]]$. Thus the global
assertion of Theorem~\ref{OC} is equivalent to the local
assertion for cyclic $p$-power extensions $k[[x]]\hra k[[z]]$.
\vskip5pt
We tackle the case of $p$-power cyclic extensions
$k[[t]]\hra k[[z]]$ as follows.
\vskip5pt
\underbar{Step 1}. Let $\uix=(\ix_1,\dots\ix_\nx)$ be a
fixed upper ramification jumps sequence. By 
Key Lemma~\ref{keylemma1} and~Theorem~\ref{charpOC}
there exists some $N$ and sequences 
$\uix_\mu=(\ix_1,\dots,\ix_{\nx_\mu})$, $1\leq\mu\leq N$, 
depending on $\uix$ only, such that the following hold: 
Let $\euo$ be a complete discrete valuation ring over 
with residue field $k$, and $x_1,\dots,x_N\in\eum_\euo$
be distinct points. Then for every $\uix$-KG-cover 
$Y\to\lvPt k$ there exists a $\lvZ/p^\nx$-cover of projective
smooth $\euo$-curves $\clY_\euo\to\lvPt\euo$ satisfying:  
\vskip2pt
a) The special fiber of $\clY_\euo\to\lvPt\euo$ is the given 
$\uix$-KG-cover $Y\to\lvPt k$.
\vskip2pt
b) The generic fiber $\clY_\kk\to\lvPt\kk$ of 
$\clY_\euo\to\lvPt\euo$ is branched above $x_1,\dots,x_N$ 
only.\footnote{\hhb2N.B., $\clY_\kk\to\lvPt\kk$ is not always
an $\uix$-KG-cover!} 
\vskip3pt
c) The upper ramification jumps above each $x_\mu$ are
$\uix_\mu:=(\ix_1,\dots,\ix_{\nx_\mu})$, $\mu=1,\dots,N$.
\vskip5pt
\underbar{Step 2}. Let $\kk\hra\akk$ be an algebraic 
closure, and $\clY_\akk \to \lvPt\akk$ be the base change 
of $\clY_\kk \to \lvPt\kk$. Then $\clY_\akk \to \lvP^1_\akk$ 
is a $\lvZ/p^\nx$-cover of projective smooth curves with
no essential ramification. 

\begin{hypo}
\label{hypothesis}
In the Notations~\ref{notalast}, suppose that for every 
$\uix_\mu=(\ix_1,\dots,\ix_{\nx_\mu})$ with $\mu=1,\dots,N$,
the subset $\Sigma_{\uix_\mu}\subset\lvA^{{\scriptscriptstyle|}
{\uix_\mu}{\scriptscriptstyle|}}$ is Zariski dense. 
\end{hypo}
For every fixed $\mu=1,\dots,N$, consider the 
$\uix_\mu$-KG-cover $Y_\mu\to\lvPt l$ of the local 
$\lvZ/p^{\nx_\mu}$-extension $\akk[[t_\mu]]\hra\akk[[z_\mu]]$,
where $z_\mu$ and $t_\mu$ are local parameters at 
$\ybf_\mu\mapsto\xbf_\mu$. Then by Proposition~\ref{specprop} 
applied for each $\uix_\mu=(\ix_1,\dots,\ix_\mu)$, 
there exists some algebraic integer $\pi_\mu$ such that
$\uix_\mu$-KG-cover $Y_\mu\to\lvPt l$ has a smooth lifting
over $W(l)[\pi_\mu]$. Thus by the local-global 
principle~Fact~\ref{equivOC}, it follows that if $\pi_\uix$
is an algebraic integer such that $\pi_\mu\in W(l)[\pi_\uix]$
for all $\mu=1,\dots,N$ then $Y_l\to\lvPt l$ has a smooth 
lifting $\clY_{\clO_1}\to\lvPt{\clO_1}$ over $\clO_1:=W(l)[\pi_\uix]$. 
\vskip2pt
Let $v_1$ be the canonical valuation of $\clO_1$,
and $v_0$ be the (unique) prolongation of the valuation 
of $\euo$ to $l$, say having valuation ring $\clO_0$. 
Then the base change $\clY_{\clO_0}\to\lvPt{\clO_0}$ of the 
$\lvZ/p^\nx$-cover of complete curves $\clY_\euo\to\lvPt\euo$
under $\euo\hra\clO_0$ is a $\lvZ/p^\nx$-cover of complete
smooth $\clO_0$-curves with generic fiber $\clY_l\to\lvPt l$.
And the $\lvZ/p^\nx$-cover of complete smooth $l$-curves 
$\clY_l\to\lvPt l$ is the special fiber of the $\lvZ/p^\nx$-cover
of smooth $\clO_1$-curves $\clY_{\clO_1}\to\lvPt{\clO_1}$. 
The setting $v:=v_0\circ v_1$ and letting $\clO$ be the 
valuation ring of $v$, it follows by Lemma~\ref{transsmooth} 
that there exists a smooth lifting of $Y\to\lvPt k$ to a
$\lvZ/p^\nx$-cover of smooth $\clO$-curves $\clY_\clO\to\lvPt\clO$.
\vskip2pt
Since the $\uix$-KG-cover $Y\to\lvPt k$ we started with was 
arbitrary, it follows that the Hypothesis~\ref{hypothesis} implies
that $\Sigma_\uix=\lvAix$. Hence by Proposition~\ref{specprop} 
we conclude that: 
\vskip7pt
\centerline{\it Hypothesis~\ref{hypothesis} implies the 
existence of an algebraic integer $\pi_\uix$ such\/}
\centerline{\it that every $\uix$-KG-cover 
$Y\to\lvPt k$ has a smooth lifting over $W(k)[\pi_\uix]$.\/}
%
%
\vskip7pt
\noindent
D) {\it Concluding the proof of the Oort Conjecture\/}
\vskip5pt
By the observation above, the proof of the Oort Conjecture
is reduced to showing that the Hypothesis~\ref{hypothesis} 
holds for every system of upper ramification indices 
$\uix=(\ix_1,\dots,\ix_\nx)$ which has no essential jump
indices, i.e., $p\hhb1\ix_{\rhal-1}\leq\ix_\rhal< p\hhb1\ix_{\rhal-1}+p$ for 
$\rhal=1,\dots,\nx$. Via the local-global principle~Fact~\ref{equivOC}, 
this fact is equivalent to a (very) special case of the local Oort 
Conjecture, which follows from a more general (but still partial) 
result recently announced by \nmnm{Obus--Wewers}~\cite{O--W}, 
see~\nmnm{Obus}~\cite{Ob},~Theorem~6.28. Here 
is the special case needed here:
\begin{keylemma} 
\label{keylemma2}
{\rm (\nmnm{Special case of Obus--Wewers})}
In notations and context as above, let $k[[t]]\hra k[[z]]$ be 
cyclic extension of degree~$p^\nx$ which has no essential 
ramification. Then the local Oort Conjecture holds for 
$k[[t]]\hra k[[z]]$, i.e., $k[[t]]\hra k[[z]]$ has a smooth
lifting over some finite extension~$R$ of $\,W(k)$ to a
smooth cyclic cover $R[[T]]\hra R[[Z]]$.
\end{keylemma}
\begin{proof}  Recall that Lemma~6.27 from~\nmnm{Obus}~\cite{Ob}
asserts that the local Oort conjecture holds for cyclic extensions
$k[[t]]\hra k[[z]]$ of degree $p^\nx\!$, provided the upper ramification 
jumps $\ux_1\leq\dots\leq \ux_\nx$ satisfy: For every $1\leq\alp< \nx$, 
there is no integer $m$ such that:
\[
\hhb{10}(*)\hhb{80}\ux_{\alp+1}-p\hhb1\ux_\alp < p\hhb1m \leq
   (\ux_{\alp+1}-p\hhb1\ux_\alp)\,
       {{\ux_{\alp+1}}_{\phantom|}\over
       {\ \ux_{\alp+1}-\ux_\alp}\ }.\hhb{100}
\]
Notice that if $k[[t]]\hra k[[z]]$ has no proper essential
ramification jumps, thus by definition 
$\ux_{\alp+1} \leq p\hhb1\ux_\alp +p-1$
for all $1\leq\alp < \nx$, then the hypothesis~$(*)$ above is 
satisfied. Indeed, we first notice that $p\hhb1\ux_\alp\leq \ux_{\alp+1}$
implies that $m$ must be positive. Hence setting 
$\delta:=\ux_{\alp+1}-p\hhb1\ux_\alp$ and $u:=\ux_\alp$, 
the second inequality becomes
$p\hhb1m\big(\delta+(p-1)u\big)\leq\delta(\delta+p\hhb1u)$,       
which is equivalent to $p(p-1)mu\leq\delta(\delta+pu-pm)$.
Hence taking into account that $\delta\leq p-1$, we get:  
$p(p-1)mu\leq(p-1)(p-1+pu-mp)$, thus dividing by $(p-1)$,
we get: $p\hhb1mu\leq p-1+p\hhb1u-mp$, or equivalently,
$p(m-1)(u+1)+1\leq0$, which does not hold for any 
positive integer~$m$. 
\end{proof}
This concludes the proof of Theorem~\ref{OC}.
\begin{remark} 
\label{remarkOW}
The main result of \nmnm{Obus--Wewers},
see~\cite{Ob},~Theorem~6.28, asserts that if the 
hypothesis~$(*)$ above is satisfied for all $3\leq\alp< \nx$, 
then the local Oort conjecture holds for the cyclic extension 
$k[[t]]\hra[[z]]$. Therefore, the Oort Conjecture {\it holds 
unconditionally for $n=3$,\/} and the above hypothesis~$(*)$ 
kicks in only for exponents $3<\nx$. Unfortunately, for $3 < \nx$, 
the hypothesis~$(*)$ becomes very restrictive indeed. 
Setting namely $\delta:=\ux_{\alp+1}-p\hhb1\ux_\alp$ and 
$u:=\ux_\alp$, we have that $\ux_{\alp+1}=\delta+p\hhb1u$ 
for some $0\leq\delta$. If $0<\delta$, then $\delta$ is 
prime to $p$ hence writing $\delta=r\hhb1p-\eta$ with 
$1\leq\eta\leq p-1$, the hypothesis~$(*)$ for $m:=r$ implies that:
\vskip5pt
\centerline{\it At least one of the inequalities \ 
$\delta < r\hhb1p\leq\delta(\delta+p\hhb1u)/\big(\delta+(p-1)u\big)\,$
           does not hold.}
\vskip5pt  
\noindent         
Since the first inequality holds, the second inequality 
must \underbar{not} hold. Therefore we must have
$r\hhb1p > \delta(\delta+p\hhb1u)/\big(\delta+(p-1)u\big)$.
Since the LHS equals $\delta+\delta u/\big(\delta+(p-1)u\big)$,
and $p\hhb1r-\delta=\eta$, the above inequality 
is equivalent to $\eta > \delta u/\big(\delta+(p-1)u\big)$,
hence to $\eta\delta+\eta(p-1)u>\delta u$, and finally to
$\eta^2(p-1)>\big(\delta-\eta(p-1)\big)(u-\eta)$. Thus 
since $\delta-\eta(p-1)=(r-\eta)p$, the last inequality 
becomes $\eta^2(p-1)>p(r-\eta)(u-\eta)$. Since $p>p-1$, 
the last inequality implies $\eta^2>(r-\eta)(u-\eta)$. 
On the other hand, one has
$1\leq\eta <p$ and $p^2\leq u$ for $1<\alp$, thus 
$\eta^2<u-\eta$. Therefore, in order to satisfy the inequality, 
one must have $r\leq\eta$. Hence the hypothesis~$(*)$
for $1<\alp< \nx$ is equivalent to:
\vskip5pt
\centerline{\ $(*)'$ \ {\it 
If $\,\ux_{\alp+1}=p\hhb1\ux_\alp+p\hhb1r_\alp -\eta_\alp$
with $\,0\leq\eta_\alp< p$, then $\,0\leq r_\alp\leq\eta_\alp$
for all $\,1 < \alp < \nx$.\/}}
\vskip5pt
\noindent
In particular, $\eta_\alp=1$ implies $\delta_\alp= p-1$, 
and $\delta_\alp\leq(p-1)^2$ in general. It seems to me 
that the reformulation~$(*)'$ is better/easier than the original 
formulation of hypothesis~$(*)$.
\end{remark}

\bibliographystyle{plain}

\end{document}